\documentclass[letterpaper,10pt,conference]{ieeeconf}
\IEEEoverridecommandlockouts
\makeatletter

\let\proof\@undefined
\let\endproof\@undefined
\makeatother

\usepackage{graphicx,url}
\graphicspath{{./figs/}}

\usepackage[algo2e,vlined,ruled,linesnumbered]{algorithm2e}
\usepackage{amsthm, amsmath, amssymb}
\usepackage{epsfig,color}
\usepackage{xspace}
\usepackage{floatrow}
\usepackage[caption=false,font=footnotesize]{subfig}

\usepackage{multirow}

\usepackage{fixltx2e}

\newtheorem{theorem}{Theorem}[section]

\newtheorem{corollary}[theorem]{Corollary}
\newtheorem{lemma}[theorem]{Lemma}

\theoremstyle{definition}
\newtheorem{definition}[theorem]{Definition}

\theoremstyle{remark}
\newtheorem{remark}[theorem]{Remark}

\newcommand{\Int}{\mathbb{Z}} \newcommand{\nat}{{\mathbb{N}}}
 \newcommand{\real}{{\mathbb{R}}}
\newcommand{\set}[1]{\ensuremath{\left\lbrace #1 \right\rbrace}}
\newcommand{\floor}[1]{\left\lfloor #1 \right\rfloor}

\DeclareMathOperator*{\argmax}{argmax}

\newcommand\oprocendsymbol{\hbox{$\bullet$}}
\newcommand\oprocend{\relax\ifmmode\else\unskip\hfill\fi\oprocendsymbol}

\urldef{\smith}\url{stephen.smith@uwaterloo.ca}
\urldef{\talha}\url{stjawaid@uwaterloo.ca}

\title{The Maximum Traveling Salesman Problem with Submodular Rewards}

\author{Syed Talha Jawaid \qquad Stephen L. Smith\thanks{This research is partially
    supported by the Natural Sciences and Engineering Research Council
    of Canada (NSERC). }
  \thanks{The authors are with the Department of Electrical and
    Computer Engineering, University of Waterloo, Waterloo ON, N2L 3G1
    Canada  (\talha; \smith)}}

\begin{document}

\maketitle

\begin{abstract}
In this paper, we look at the problem of finding the tour of maximum reward on an 
undirected graph where the reward is a submodular function, that has a curvature
of $\kappa$, of the edges in the tour. This 
problem is known to be NP-hard.  We analyze two simple algorithms for finding an 
approximate solution.  Both algorithms require $O(|V|^3)$ oracle calls to the submodular 
function.  The approximation factors are shown to be $\frac{1}{2+\kappa}$ and
$\max\set{\frac{2}{3(2+\kappa)},\frac{2}{3}(1-\kappa)}$, respectively;
so the second method has better bounds for low values of $\kappa$.
We also look at how these algorithms perform for a directed graph 
and investigate a method to consider edge costs in addition to rewards.  The problem has 
direct applications in monitoring an environment using autonomous mobile sensors where the 
sensing reward depends on the path taken.  We provide simulation results to empirically 
evaluate the performance of the algorithms.
\end{abstract}

\section{Introduction}

The maximum weight Hamiltonian cycle is a classic problem in
combinatorial optimization. It consists of finding a cycle in a graph
that visits all the vertices and maximizes the sum of the weights
(i.e., the reward) on the edges traversed.  Also referred to as the
max-TSP, the problem is NP-hard and so no known polynomial time
algorithms exists to solve it. However, a number of approximation
schemes have been developed.  In~\cite{MLF-GLN-LAW:79} four simple
approximation algorithms are analysed.  The authors show that greedy,
best-neighbour, and 2-interchange heuristics all give a $\frac{1}{2}$
approximation to the optimal tour.  They also show that a matching
heuristic, which first finds a perfect 2-matching and then converts
that to a tour, gives a $\frac{2}{3}$ approximation.
In~\cite{RH-SR:98}, the authors point out that Serdyukov's
algorithm--- an algorithm which computes a tour using a combination of
a maximum cycle cover and a maximum matching---can give a
$\frac{3}{4}$ approximation.  They also give a randomized algorithm
that achieves a $\frac{25}{33}$ approximation ratio.  In this paper we
look at extending the max-TSP problem to the case of submodular
rewards.

The main property of a submodular function is that of decreasing
marginal value, i.e., choosing to add an element to a smaller set will
result is a larger reward than adding it later.  One application in
which submodular functions appear is in making sensor measurements in
an environment.  For example, in~\cite{CG-AK-AS:05} the authors
consider the problem of placing static sensors over a region for
optimal sensing.  If a sensor is placed close to another, then the
benefit gained by the second sensor will be less that if the first
sensor had not already been placed.  This can be represented
quantitatively by using the concept of mutual information of a set of
sensors, which is a submodular function.  Other areas where submodular
functions come up include viral marketing, active
learning~\cite{DG-AK:11} and AdWords assignment \cite{PRG-ASS:07}.  A
different form of sensing involves using mobile sensors for persistent
monitoring of a large environment using a mobile robot \cite{AS-AK-CG-WK:09}.
The metric used to determine the quality of the sensing is usually submodular in
nature. Due to
the persistent operation, it is desirable to have a closed walk or a
tour over which the sensing robot travels. This motivates the problem of
finding a tour that has the maximum reward.

Various results exist for maximizing a monotone submodular function
over an independence system constraint.  This problem is known to be
NP-hard, even though minimization of a submodular function can be
achieved in polynomial time (\cite{AS:00},\cite{SI-LF-SF:01}).
Approximation bounds exist for optimizing over a uniform matroid
\cite{GLN-LAW-MLF:78A}, any single matroid \cite{GC-CC-MP-JV:11}, an
intersection of $p$ matroids and, more generally, $p$-systems
\cite{GLN-LAW-MLF:78B} as well as for the class of $k$-exchange
systems \cite{JW:12}. Some bounds that include
the dependence on curvature are evaluated in \cite{MC-GC:84}.

%%%% 
\emph{Contributions:} The contributions of this paper are to present
and analyze two simple algorithms for constructing a maximum tour on a
graph.  The metric used in maximizing the ``reward" of a particular
tour is a positive monotone submodular function of the edges.
We frame this problem as an optimization over an
independence system constraint.  The first method is greedy and is
shown to have a $\frac{1}{2+\kappa}$ approximation. The second method creates
a 2-matching and then turns it into a tour. This gives a
$\max\set{\frac{2}{3(2+\kappa)},\frac{2}{3}(1-\kappa)}$ worst case approximation
where $\kappa$ is the curvature of the submodular function.  Both
techniques require $O(|V|^3)$ value oracle calls to the submodular
function.  The algorithms are also extended to directed graphs.
To obtain these results, we present
a new bound for the greedy algorithm as a function of curvature.
We also present some preliminary results for the case of a
multi-objective optimization consisting of submodular (sensing)
rewards on the edges along with modular (travel) costs.  We
incorporate these two objectives into a single function, but it is no
longer monotone nor is it positive.  We provide bounds on the
performance of our algorithms in this case, but they depend on the
relative weight of the rewards.

%%%
\emph{Organization:} The organization of this paper is as follows.  In
Section~\ref{sec:background} we review some material on independence
systems, submodularity, graphs and approximation methods for
submodular functions.  In Section~\ref{sec:problem} we formalize our
problem.  In Section~\ref{sec:greedy} we analyze a simple greedy
strategy.  In Section~\ref{sec:matching} we present and analyze a
strategy to construct a solution using a matching.  In
Section~\ref{sec:directedExtension} we look at how the presented
algorithms extend to the case where the graph is directed.  Finally,
in Section~\ref{sec:incCosts} we discuss a method to incorporate costs
into the optimization.  Some simulation results are provided in
Section~\ref{sec:sims} comparing the given strategies for various
scenarios.

\section{Preliminaries}
\label{sec:background}

Here we present preliminary concepts and give a brief summary of results on
combinatorial optimization problems over independence systems.

\subsection{Independence systems}

Combinatorial optimization problems can often be formulated as the maximization or minimization
over a set system $(E,\mathcal{F})$ of a cost function $f:\mathcal{F} \to \real$,
where $E$ is the base set of all
elements and ${\cal  F}\subseteq 2^E$. 
An \textbf{independence system} is a set system that is closed under subsets (i.e., if $A \in \mathcal{F}$ then $B \subseteq A \implies B\in {\cal F}$).  Sets in $\cal F$ are referred to as ``independent sets".  The set of maximal independent sets (i.e., all $A \in {\cal F}$ such that $A \cup \{x\} \notin {\cal F}, \forall x \in E \setminus A$) are the \textbf{bases}.

\begin{definition}[$p$-system]
Given an independence system $S = (E,{\cal F})$.
For any $A \subseteq E$ let
\begin{align*}
U(A) &:=
\max_{ \set{B: B \text{ is a basis of } A} } |B|
\\
L(A) &:=
\min_{ \set{B: B \text{ is a basis of } A} } |B|
\end{align*}
be the sizes of the 
maximum and minimum cardinality bases of $A$ respectively.
For $S$ to be a $p$-system,
\[ U(A) \leq p L(A), \forall A \subseteq E. \]
\end{definition}

\begin{definition}[$p$-extendible system] %
  \label{def:p_extend} %
  An independence system $(E,{\cal F})$ is $p$-extendible if given an independent set $B \in {\cal F}$, for every subset $A$ of $B$ and for every $x \notin A$ such that $A\cup\set{x} \in {\cal F}$, there exists $C \subseteq B \setminus A$ such that $|C| \leq p$ and for which $(B\setminus C)\cup\set{x} \in {\cal F}$.
\end{definition}

\begin{remark}
A $p$-extendible system is a $p$-system.
\oprocend
\end{remark}

\begin{definition}[Matroid]
  An independence system $(E, \mathcal{F})$ is a matroid if it satisfies the
  additional property:
\begin{itemize}
\item If $X,Y \in \mathcal{F}$ and $|X| > |Y|$, then $\exists x \in X \backslash Y$ with $Y \cup \{x\} \in \mathcal{F}$
\end{itemize}
\end{definition}

\begin{remark}
A matroid system is a $1$-extendible system.
\oprocend
\end{remark}
\begin{remark}
Any independence system can be represented as the intersection of a finite
number of matroids \cite{BK-JV:07}.
\oprocend
\end{remark}

An example of a matroid is the \textit{partition matroid}.
The base set is the union of $n$ disjoint sets, i.e.
$E = \bigcup_{i=1}^n E_i$ where $E_i \cap E_j = \emptyset$ for $i\neq j$.
Also given $k \in \Int_+^n$.
The matroid is
${\cal F} := \set{A \subseteq E: |A\cap E_i| \leq k_i, \forall i =1\ldots n}$.

\subsection{Submodularity}

Without any additional structure on the set-function $f$,
the optimization problem is generally intractable.  However, a fairly
general class of cost functions for which approximation
algorithms exist is the class of submodular set functions.  

\begin{definition}[Submodularity]
  Let $N$ be a finite set.  A function $f:2^N \to \real$ is submodular
  if
  \[ 
  f(S) + f(T) \geq f(S \cup T) + f(S \cap T), \quad \forall S,T
  \subseteq N.
  \]
\end{definition}

Submodular functions satisfy the property of \textit{diminishing
  marginal returns}. That is, the contribution of any element $x$ to 
the total value of a set can only decrease as the set gets bigger.
More formally, let $\Delta_A(B) := f(A \cup B) - f(A)$. Then,
\[  
\Delta_A(x) \geq \Delta_B(x),
\quad \forall A \subseteq B \subseteq N.
\]

Since the domain of $f$ is $2^N$, there are an exponential number of possible values
for the set function. As a result, enumerating the value of every single subset of the 
base set is not an option.
We will assume that $f(S)$ for any $S \subseteq N$ is determined by a black box function.
This \textit{value oracle} is assumed to run in polynomial time in the size of the input set.

The class of submodular functions is fairly broad and includes linear functions.
One way to measure the degree of submodularity is the \textit{curvature}.
A submodular function is said to have a curvature of $\kappa \in [0,1]$ if for any
$A \subset N$ and $e \in N \setminus A$
\begin{align}
\label{eqn:curvature}
\Delta_A(e) \geq (1-\kappa) f(e).
\end{align}
In other words, the minimum possible marginal benefit of any element $e$ is
within a factor of $(1-\kappa)$ of its maximum possible benefit.

We formulate a slightly stronger notion of the curvature - the independence system
curvature, $\kappa_I$ - by taking the independence system in to account. 
In this case, \eqref{eqn:curvature} need only be satisfied for any
$A \in {\cal F}$ and $e \in N \setminus A, A \cup \{e\} \in {\cal F}$.
This value of curvature will be lower than the one obtained by 
the standard definition given above.

\subsection{Greedy Algorithms}
The greedy algorithm is a simple and well-known method for finding solutions to
optimization problems. The basic idea is to choose the ``optimal" element at each
step. So given a base set of elements $E$, the solution $S$ is constructed as:
\begin{enumerate}
\item Pick the best element $e$ from $E$.
\item If $S \cup \set{e}$ is a feasible set, then $S = S\cup \set{e}$.
\item $E = E\setminus \set{e}$.
\item Repeat until $S$ is maximal or $E$ is empty.
\end{enumerate}

A general greedy algorithm for maximizing a submodular function over an
independence system is given in Algorithm~\ref{alg:generalGreedy}.
Since the objective function is submodular, the marginal value of each element in the
base set changes in every iteration and so has to be recalculated.
This causes the runtime
to be $O(|N|^2f)$ where $f$ is the runtime of the value oracle.

\begin{algorithm2e}[htb]
\KwIn{An independence system $M = (E,{\cal F})$.
	A function $f:{\cal F} \to \real$.}
\KwOut{Basis of ${\cal I}$.}
$N \leftarrow E$\;
\While{$N \neq \emptyset$ {\bf and} $S$ not maximal in $\cal F$}
{
	\ForEach{$e \in N$}{
		calculate $\Delta(e) := f(S\cup\set{e}) - f(S)$\;
	}
	$m := \argmax_e \Delta(e)$\;
	\lIf{$S \cup \set m \in {\cal F}$}{
		$S \leftarrow S \cup \set{m}$\;
	}
	$N \leftarrow N \setminus \set m$\;
}
\Return S\;
\caption{generalGreedy$\big( (E,{\cal F}), f \big)$}
\label{alg:generalGreedy}
\end{algorithm2e}

Based on the properties of submodularity, a more efficient implementation can be constructed.
The specific property that helps here is that of decreasing marginal benefit.
Given $A \subset B$ and two elements $e_1, e_2 \notin B$.
If $\Delta_A(e_1) \leq \Delta_B(e_2)$,
then we can conclude that $\Delta_B(e_1) \leq \Delta_B(e_2)$.
Therefore, $\Delta_B(e_1)$ does not need to be calculated.

This idea was first proposed in \cite{MM:78}.
Algorithm~\ref{alg:generalGreedyAcc} is a modified version
of the accelerated greedy algorithm presented in \cite{DG-AK:11} and takes into account the independence constraint.
The main modification from \cite{DG-AK:11} is the check for independence in line~\ref{algln:accGreedyCheckInd}.

\begin{algorithm2e}[htb]
\KwIn{An independence system $(E,{\cal F})$.
	A function $f:{\cal F} \to \real$.}
\KwOut{Basis of ${\cal F}$.}
Priority Queue $Q \leftarrow \emptyset$\;
\lForEach{$e \in E$}{$Q$.insert($e$,$\infty$)\;}
\While{$Q \neq \emptyset$ {\bf and} $S$ not maximal in $\cal F$}
{
	$e_{max} \leftarrow$ NULL;
	$\delta_{max} \leftarrow -\infty$\;
	\While{$Q \neq \emptyset$ {\bf and} $\delta_{max} < Q$.maxPriority()}
	{
		$e \leftarrow Q$.pop()\;
		\If{$T \cup \set e \in {\cal F}$ \label{algln:accGreedyCheckInd}}
		{
			$\Delta(e) \leftarrow f(T\cup\set{e}) - f(T)$\;
			$Q$.insert($e$,$\Delta(e)$)\;
			\If{$\delta_{max} < \Delta(e)$}
			{
				$\delta_{max} \leftarrow \Delta(e)$;
				$e_{max} = e$\;
			}
		}
	}
	\If{$Q \neq \emptyset$}
	{
		$S \leftarrow S \cup \set{e_{max}}$\;
		$Q$.remove($e_{max}$)\;
	}
}
\Return S\;
\caption{generalGreedy$\big( (E,{\cal F}), f \big)$ - Accelerated}
\label{alg:generalGreedyAcc}
\end{algorithm2e}

Unfortunately, the accelerated greedy algorithm has the same worst case bound as
the naive version.
However, empirical results have shown that it can achieve significant speedup
factors \cite{MM:78},\cite{DG-AK:11}.

\subsection{Approximations}
Here we consider some useful results for optimization over independence systems.
First, lets look at the case of a linear objective function.
%For a $p$-system where all the elements of the base set have been assigned a weight,
%the problem of finding a maximum weight basis can be approximated to a factor of $\frac{1}{p}$
%by a greedy algorithm \cite{BK-JV:07} \cite{DH-BK-TAJ:80} \cite{TAJ:79}.
In the special case of a matroid ($p=1$), the greedy algorithm gives an optimal
solution \cite{BK-JV:07}.

\begin{lemma}
\label{lem:greedylnrpsystem}
For the problem of finding the basis of a $p$-system that maximizes some linear
non-negative function, a greedy algorithm gives a $\frac{1}{p}$ approximation.
\end{lemma}

In~\cite{GLN-LAW-MLF:78A} the authors look at maximizing a submodular
function over a uniform matroid (selecting $k$ elements from a set). 
They show that the greedy algorithm gives a worst case approximation
of $1 - \frac{1}{e}$.
This is actually the best factor that can be achieved as in \cite{UF:98} 
it is shown that to obtain a $(1 - \frac{1}{e} + \epsilon)$-approximation for any $\epsilon > 0$
is NP-hard for the maximum $k$-cover problem, which is the special case of a
uniform matroid constraint.

In \cite{GLN-LAW-MLF:78B}, the optimization problem is generalized to
an independence system represented as the intersection of $P$ matroids.
The authors state that the result can be extended to $p$-systems.
A complete proof for this generalization is given in \cite{GC-CC-MP-JV:11}.
For a single matroid constraint, an algorithm to obtain a $(1-1/e)$ approximation
is also given in \cite{GC-CC-MP-JV:11}.

\begin{lemma}
\label{lem:greedysubmodpsystem}
For the problem of finding the basis of a $p$-system that maximizes some monotone
non-negative submodular function, a greedy algorithm gives a $\frac{1}{p+1}$
approximation.
\end{lemma}

We know that linear functions (i.e. curvature is 0) are a special case of
submodular functions so it is reasonable to expect
the greedy bound to be a continuous function of the curvature.
In \cite{MC-GC:84}, bounds that include the curvature are presented for
single matroid systems, uniform matroids and independence
systems. For a system that is the intersection of $p$ matroids the greedy bound
is shown to be $\frac{1}{p+\kappa}$.
In the following we extend this result to $p$-systems.

\begin{theorem}
\label{thm:greedycurvsubmodpsystem}
Consider the problem of maximizing a monotone submodular function $f$
with curvature $\kappa$, over a $p$-system.
Then, the greedy algorithm gives an approximation factor of $\frac{1}{p+\kappa}$.
\end{theorem}
\begin{proof}
The proof is inspired from the proof where the system is the intersection
of $p$ integral polymatroids \cite[Theorem~6.1]{MC-GC:84}.
Let $W$ be the optimal set.
Let $S = \set{s_1,\ldots,s_k}$ be the result of the greedy algorithm where the
elements are enumerated in the order in which they were chosen.
For $t = 1,\ldots,k$, let $S_t := \bigcup_{i=1}^t s_i$
and $\rho_t := \rho_{S_{t-1}}(s_t) = f(S_t)- f(S_{t-1})$.
Therefore, $f(S) = \sum_{t=1}^k \rho_t$.
So,
\begin{align}
f(W \cup S) &\geq f(W) + \sum_{e \in S} \rho_{(W \cup S) \setminus e}(e)
\notag\\
&\geq f(W) + \sum_{e \in S} (1-\kappa)f(e)
\notag\\
&\geq f(W) + (1-\kappa)\sum_{t=1}^k \rho_t, \label{eqn:lwrBound}
\end{align}
where the first and third inequalities hold due to submodularity and the second
is by the definition of curvature.
Also,
\begin{align}
f(W \cup S) \leq f(S) + \sum_{e\in W \setminus S}\rho_S (e).
\label{eqn:uprBound1}
\end{align}
Following from the analysis of the greedy algorithm for $p$-systems in
\cite[Appendix B]{GC-CC-MP-JV:11}, a $k$-partition $W_1,\ldots,W_k$ of $W$,
can be constructed such that $|W_i| \leq p$ and
$\rho_t \geq \rho_{S_{t-1}}(e)$ for all $e \in W_t$.
Therefore,
\begin{align}
\sum_{e\in W\setminus S}\rho_S (e)
&= \sum_{i=1}^k \sum_{e\in W_i\setminus S}\rho_S (e)
\notag
\\
&\leq \sum_{i=1}^k |W_i\setminus S| \rho_t
\leq p \sum_{i=1}^k \rho_t, \label{eqn:uprBound2}
\end{align}
since $\rho_S (e) \leq \rho_{S_{t-1}}(e)$ for all $t$ by submodularity.

Combining \eqref{eqn:uprBound1} and \eqref{eqn:uprBound2} with \eqref{eqn:lwrBound},
the desired result can be derived as follows,
\begin{align*}
&f(W) + (1-\kappa)f(S) \leq f(S) + p f(S)
\\
\implies &f(W) \leq (p+\kappa) f(S).
\end{align*}
\end{proof}

\subsection{Set Systems on Graphs}
In this section we introduce some graph constructs and then give some results 
relating them to $p$-systems.

We are given a graph $G = (V, E, c, w)$ where $V$ is the set of vertices and $E$
is the set of edges.
Each of the edges has a cost given by the function $c: E \to \real_{\geq 0}$.
A set of edges has a cost of $c(S) = \sum_{q \in S} c(q)$.
A set of edges also has a reward or utility associated with it given by
the submodular function $w:2^E \to \real_{\geq 0}$. 

A \textit{walk} is a sequence of vertices $(v_1,\ldots, v_n)$ such that
$e_i := \{v_i,v_{i+1}\} \in E$.
A \textit{path} is a walk such that $e_i \neq e_j, \forall i,j$.
A \textit{simple path} is a path such that
$v_i \neq v_j, \forall i,j = 1 \ldots n$.
A \textit{simple cycle} is a path such that
$v_i \neq v_j, \forall i,j = 1 \ldots n-1$,
and $v_n = v_1$.
A \textit{Hamiltonian cycle} is a simple cycle that visits all the vertices of $G$.
We will refer to a Hamiltonian cycle as a \textbf{tour},
and a simple cycle that is not a tour as a \textbf{subtour}.
Let $\mathcal{H} \subset 2^E$ be the set of all tours in $G$.

\subsubsection{b-matching}
Given an undirected graph $G = (V,E)$.
Let $\delta(v)$ denote the set of edges that are incident to $v$.
\begin{definition}[$b$-matching]
Given edge capacities
$u:E[G] \to \nat \cup \infty$ and vertex capacities
$b:V[G] \to \nat$.
A {\bf b-matching} is an assignment to the edges $f:E[G] \to \Int_+$
such that
$f(e) \leq u(e), \forall e \in E[G]$ and
$\sum_{e \in \delta(v)} f(e) \leq b(v), \forall v \in V[G]$.
\\
If $\sum_{e \in \delta(v)} f(e) = b(v), \forall v$,
the b-matching is {\bf perfect}.
\\
A {\bf simple} b-matching is the special case that the edge capacity $u(e) = 1$
for each $e \in E$.
\end{definition}
For the rest of this paper, any reference to a b-matching will
always refer to a simple b-matching.
\begin{theorem}[Mestre, \cite{JM:06}]
A simple $b$-matching is a $2$-extendible system.
\end{theorem}

\subsubsection{The Traveling Salesman Problem}
Given a complete graph, the classical Traveling Salesman Problem (TSP) is to find a 
minimum cost tour.
The TSP can be divided into two variants: the Asymmetric TSP and the Symmetric TSP.
In the ATSP, for two vertices $u$ and $v$, the cost of edge $(u,v)$ is different from 
the cost of $(v,u)$, which amounts to the graph being directed.
In the STSP, $c(u,v) = c(v,u)$, which is the case if the graph in undirected.

In order to formulate the TSP, the set of possible solutions can be defined using an
independence system.
The base set of the system is the set of edges in the complete graph.
For the ATSP, a set of edges is independent if they form a collection of 
vertex disjoint paths, or a complete Hamiltonian cycle.
\begin{theorem}[Mestre, \cite{JM:06}]
The ATSP independence system is 3-extendible.
\end{theorem}

The ATSP can be formulated as the intersection of 3 matroids.
These are:
\begin{enumerate}
\item Partition matroid:
Edge sets such that the in-degree of each vertex $\leq 1$
\item Partition matroid:
Edge sets such that the out-degree of each vertex $\leq 1$
\item The 1-graphic matroid:
the set of edges that form a forest with at most one simple cycle.
\end{enumerate}

The STSP is just a special case of the ATSP.
Therefore, the results from the ATSP carry over to the STSP.
Formulating it as an ATSP, however, requires doubling the edges in 
an undirected graph.
Instead, we can directly define an independence system for the STSP.

A set of edges is independent (i.e. belongs to the collection ${\cal F}$) 
if the induced graph is composed of a collection of vertex disjont simple paths or
a Hamiltonian cycle.
This can also be characterized as the following two conditions:
\begin{enumerate}
\item Each vertex has degree at most 2
\item There are no subtours
\end{enumerate}

\begin{theorem}
The undirected TSP independence system is 3-extendible
\end{theorem}
\begin{proof}
To show this, we can consider all the cases to show that the system satisfies the
definition of a 3-extendible system.
Specifically, assume some given set $A \subset B \in {\cal F}$ and determine the number
of edges that will need to be removed from $B \setminus A$ so that adding any 
$x = \set{u,v} \notin B$ (such that $A \cup x \in {\cal F}$) to $B$ will maintain independence.

Adding an edge can violate the degree constraint on at most two vertices (specifically $u$ and
$v$) and/or the subtour constraint. To satisfy the degree constraint, at most one edge will
need to be removed from $B$ for each vertex. To satisfy the subtour requirement, at most one
edge will need to be removed from the subtour in order to break the cycle. Therefore, up to
three edges will have to be removed in total which means that the system is 3-extendible.

One case where exactly three edges will have to be removed comes about if $A$ contains 
an edge, $e_1$, incident to $u$ and another, $e_2$, to $v$.
If adding $x$ to $B$ violates both conditions of independence then we know there exists a
path $P = u \leadsto v \subseteq B$. Assume that both $e_1, e_2 \in P$.
Then one edge will have to be removed from $P$ to break the cycle (produced by adding $x$)
and two more will need to be removed to satisfy the degree requirements at $u$ and $v$.
\end{proof}

Since the STSP system is 3-extendible, it is also a 3-system.
A better result is given in the following lemma.

\begin{lemma}[Jenkyns, \protect{\cite{TAJ:79}}]
\label{lem:psysTSP}
On a graph with $n$ vertices the undirected TSP is a $p$-system with
$p = 2 - \floor{\frac{n+1}{2}}^{-1} < 2$.
\end{lemma}

\section{Problem Formulation}
\label{sec:problem}
Given a complete graph $G = (V,E,w)$, where $w$ is a submodular rewards function
that has a curvature of $\kappa$,
we are interested in analysing simple algorithms to find a
Hamiltonian tour that has the maximum reward
The specific situation we look at is
\begin{align}
\label{eqn:maxTour}
&\max_{S \in {\cal H}} w(S).
\end{align}
In Section~\ref{sec:incCosts}, we will briefly discuss the problem of where costs are
incorporated into the optimization problem.

In the following sections, we look at two methods of approximately finding the
optimal tour according to \eqref{eqn:maxTour}.

\section{A Simple Greedy Strategy}
\label{sec:greedy}
A greedy algorithm to construct the TSP is given in
Algorithm~\ref{alg:greedyTour}.
The idea is to pick the edge that will give the largest marginal benefit
at each iteration.
The selected edge
cannot cause the degree of any vertex to be more than 2 nor create any subtours.
If it fails either criteria, the edge is discarded from future consideration.

\begin{theorem}
\label{thm:greedyTourBounds}
The complexity of the greedy tour algorithm (Alg.~\ref{alg:greedyTour})
is $O(|V|^3(f + \log |V|))$, where $f$ is the runtime of the oracle,
and is a $\frac{1}{2 + \kappa}$ approximation.
\end{theorem}
\begin{proof}
By Lemma~\ref{lem:psysTSP} and Theorem~\ref{thm:greedycurvsubmodpsystem},
Algorithm~\ref{alg:greedyTour} is a $\frac{1}{2+\kappa}$-approximation of
\eqref{eqn:maxTour}.

The calculation and selection of the element of maximum marginal benefit (line~\ref{algln:recalcLoop})
requires calculating the marginal benefit for each edge in $E \setminus M$.
Note that recalculation of the marginal benefits need only be done when the set $M$ is changed.
Since only one edge is added to the tour $M$ at each update in line~\ref{algln:updateSetM},
recalculation only needs to take place a total of $|V|$ times.
In addition, the edges will need to be sorted every time a recalculation is performed 
so that future calls to find the maximum benefit element can be done in constant time
(if no recalculation is needed).
Therefore, the recalculation and sorting take $O(|V|(|E|f + |E|\log |E|))$.
The runtime for DFS is $O(|V|+|E|)$, but the DFS in run using
only the edges picked so far, so its runtime becomes $O(|V|)$.
Assuming all the other commands take $\Theta(1)$ time,
the total runtime is $O(|V|(|E|f + |E|\log|E|) + |E|(|V|+1))$.
For a complete graph, $|E| = O(|V|^2)$ and therefore
the runtime becomes 
$O(|V|^3f + 2|V|^3 \log |V| + |V|^3)$.
\end{proof}
\begin{remark}
A more efficient data structure would be to use disjoint-sets
for the vertices.
Each set will represent a set of vertices that are in the same subtour.
This gives a total runtime of $|V|^2 \log |V|$ following
the analysis in \cite[Ch.\ 21,23]{THC-CEL-RLR-CS:01}.
Unfortunately, the ``recalculation" part still dominates
the total runtime and so the resulting bound is the same.
\oprocend
\end{remark}

\begin{algorithm2e}
\KwIn{Graph $G = (V,E)$. Function oracle $w:2^E \to \real_{\geq 0}$}
\KwOut{Edge set $M$ corresponding to a tour.}
$M \leftarrow \emptyset$\;
$vDeg \leftarrow {\bf 0}_{1 \times |V|}$\;
$reCalc \leftarrow true$\;
\While{$E \neq \emptyset$ {\bf and} $|M| < |V|$}
{
	$e_m \leftarrow \argmax_{e \in E} \rho_e$\; \label{algln:recalcLoop}
	$\set{u,v} \leftarrow V[e_m]$\;
	
	\tcp{Check if edge is valid}
	$addEdge \leftarrow (vDeg[u] < 2$ {\bf and} $vDeg[v] < 2)$\;
	\If{$addEdge$ {\bf and} $|M| < |V| - 1$}
	{
		\tcp{Check for potential subtour}
		Run DFS on $G_T = (V,M)$ starting at vertex $u$\;
		\lIf{vertex $v$ is hit}
		{
			$addEdge \leftarrow false$\;
		}
	}
	
	\If{addEdge}
	{
		$S \leftarrow M \cup \set{e_m}$\; \label{algln:updateSetM}
		Increment $vDeg[u]$ and $vDeg[v]$\;
	}
	$E \leftarrow E \setminus \set{e_m}$\;
}
\Return M\;
\caption{Greedy algorithm for TSP}
\label{alg:greedyTour}
\end{algorithm2e}

Motivated by the reliance of the bound on the curvature, in the next section we will consider
a method to obtain improved bounds for functions with a lower curvature.

\section{2-Matching based tour}
\label{sec:matching}

Another approach to finding the optimal basis of an undirected TSP set system
is to first relax the ``no subtours" condition.
The set system defined by the independence condition that each vertex can have a
degree at most 2 is in fact just a simple 2-matching.
As before, finding the optimal 2-matching for a submodular function is a NP-hard problem.
We discuss two methods to approximate a solution. The first is a greedy approach and the 
second is by using a linear relaxation of the submodular function.
We will see that the bounds with linear relaxation will be better than the greedy
approach for certain values of curvature.

\subsection{Greedy 2-Matching}
One way to approximate the solution to the problem of maximizing a submodular function
to finding a maximum 2-matching is to use a greedy approach.

\begin{algorithm2e}
\KwIn{Graph $G = (V,E)$. Function oracle $f:2^E \to \real_{\geq 0}$}
\KwOut{A simple 2-matching $M$}
$M \leftarrow \emptyset$\;
$vDeg \leftarrow {\bf 0}_{1 \times |V|}$\;
$reCalc \leftarrow true$\;
\While{$E \neq \emptyset$}
{
	$e_m \leftarrow \argmax_{e \in E} \rho_e$\;
	$\set{u,v} \leftarrow V[e_m]$\;
	
	\tcp{Check if edge is valid}
	\If{$vDeg[u] < 2$ {\bf and} $vDeg[v] < 2$}
	{
		$S \leftarrow M \cup \set{e_m}$\;
		Increment $vDeg[u]$ and $vDeg[v]$\;
	}
	$E \leftarrow E \setminus \set{e_m}$\;
}
\Return M\;
\caption{Greedy algorithm for Maximal Matching}
\label{alg:greedyMatching}
\end{algorithm2e}

\begin{theorem}
\label{thm:greedyMatchingBounds}
The complexity of the greedy matching algorithm (\ref{alg:greedyMatching})
is $O(|V|^3(f + \log |V|))$, where $f$ is the runtime of the oracle.
The greedy approach is a $\frac{1}{2+\kappa}$-approximation.
\end{theorem}
\begin{proof}
Similar to the greedy tour, picking the edge of maximum benefit 
requires requires recalculation of all the marginal benefits 
and only need to be done $|V|$ times.
So picking the best edge requires $O(|V|(|E|f + |E|\log |E|))$ time.
All other parts require $\Theta(1)$ time per iteration
(for $|E|$ iterations).
Therefore, the total runtime is $O(|V|^3f+ 2|V|^3\log |V|)$.

Since a simple 2-matching is a 2-extendible system,
the greedy solution will be within $\frac{1}{2+\kappa}$ of the optimal
by Theorem~\ref{thm:greedycurvsubmodpsystem}.
\end{proof}

\subsection{Maximum 2-Matching Linear Relaxation}
For a linear objective function, the problem of finding a maximum weight
2-matching can be formulated as a binary integer program.
Let $x = \set{x_{ij}}$ where $1 \leq i < j \leq |V|$ and let
each edge be assigned a real positive weight given by $\tilde{w}_{ij}$.
Define $E(x)$ as the set of edges for which $x_{ij} = 1$.
Then the maximum weight 2-matching, $(V,E(x))$, can be obtained by solving
\begin{align*}
&\max \sum_{i=1}^{|V|-1} \sum_{j>i} \tilde{w}_{ij}x_{ij}
\\
\text{s.t. }
&\sum_{j>i} x_{ij} + \sum_{j<i} x_{ji} = 2, \quad \forall i \in \set{1,\ldots,|V|}
\\
&x_{ij} \in \set{0,1}, \quad 1 \leq i < j \leq |V|.
\end{align*}

This method, however, does not have any good bounds on runtime.
Alternatively, for a weighted graph the maximum weight 2-matching can be found in
$O(n^3)$ time \cite{BK-JV:07} via an extension of Edmonds' Maximum Weighted
Matching algorithm.

For our original problem with \eqref{eqn:maxTour} as the objective
function for the maximization, the two methods described here can obviously not
be applied directly.
Therefore, we define a linear relaxation $\tilde{w}$ of the submodular function
$w$ as follows,
\begin{align}
\label{eqn:lnrRelaxWeights}
\tilde{w}(S) = \sum_{e\in S}w(e) = \sum_{e \in S} \Delta_\emptyset(e),
\quad \forall S \subseteq E.
\end{align}
In other words, we give each edge a value that is the maximum possible
marginal benefit that it can have.
Using this relaxation, the optimal 2-matching based on the weights $\tilde{w}$
can be calculated.

\begin{theorem}
\label{thm:matchingLnrApprox}
Let $M_1$ be the maximum 2-matching for a submodular rewards function $w$ and
let $M_2$ be the maximum weight 2-matching using $\tilde{w}$ as the edge weights.
If $w$ has an independence system curvature of $\kappa_I$, then
\[
w(M_2) \geq (1-\kappa_I) w(M_1).
\]
\end{theorem}
\begin{proof}
The definition of curvature states that $\Delta_S(e) \geq (1-\kappa_I)w(e)$ for any
independent subset $S$ of the edges and $e$ such that $S \cup \{e\}$ is also
independent.
Using this and the definition of submodularity,
\begin{align*}
w(M_2) &= \Delta_\emptyset(e_1) +\Delta_{\{e_1\}}(e_2) + 
\Delta_{\{e_1,e_2\}}(e_3) +\ldots
\\
& \geq (1-\kappa_I) \sum_{e \in M_2}w(e).
\\
& = (1-\kappa_I) \tilde{w}(M_2).
\end{align*}
Since $\tilde{w}(M_2)$ is maximum, $\tilde{w}(M_2) \geq \tilde{w}(M_1)$.
Therefore,
\begin{align*}
w(M_2) \geq (1-\kappa_I)\tilde{w}(M_1)
\geq (1-\kappa_I)w(M_1),
\end{align*}
due to decreasing marginal benefits.
Note that this bound also holds using the standard definition of curvature.
\end{proof}

\subsection{Reduced 2-Matching}
The output of either of the two algorithms described will be a basis of the 2-matching
system. Once a maximal 2-matching has been obtained,
it needs to be converted into a tour.
The edge set corresponding to the 2-matching can be divided into a collection 
of disjoint sets.
These sets will either be subtours or simple paths.
Any simple path can be at most one edge in length
since otherwise its endpoints could be joined together (as the graph is complete)
contradicting the maximality of the matching.
In addition, only one of the disjoint sets will correspond to a simple path and it will contain either one edge or a single vertex.
Therefore, a maximal 2-matching will consist of a collection of vertex
disjoint subtours and at most one extra edge.

In order to convert the maximal 2-matching to a tour, the subtours will have to be 
broken by removing an edge from each one. The remaining set of simple paths
will then need to be connected up.

We first give a result on efficiently finding a subset to remove from a set while 
maintaining $\frac{2}{3}$ of the value. We then give an algorithm to reduce a set 
of subtours starting from a maximal 2-matching.

\subsubsection{Removing elements from a set}
Given a set $S$  and a $m$-partition of the set $\set{A^i}_{i=1}^m$,
i.e. $S= \bigcup_{i=1}^m A^i$
and that $A^i \cap A^j = \emptyset$ for all $i \neq j$.
Each part contains $n_i$ elements, $A^i = \bigcup_{j=1}^{n_i} a_j^i$,
such that $3 \leq n_i \leq N$.
Let $k = \min_i n_i$ and let $\bar{n} = (n_1, \ldots, n_m)$.
Also given a monotone non-decreasing submodular function $f:2^S \to \real_{\geq 0}$.
Let $A_{-j}^i := A^i \setminus a_{j}^i$.
\begin{theorem}
\label{thm:existsSmallSubset}
Given set $S$ and disjoint subsets $A^i, i = 1,\ldots,m$, where $k$ is the size of the 
smallest subset $A^i$, as defined above.
There exists a set $R$ of $m$ elements (to which each set $A^i$ contributes exactly one element)
such that $f(R)\geq (1 - \frac{1}{k}) f(S) \geq \frac{2}{3}f(S)$.
\end{theorem}
\begin{proof}
\newcommand{\veci}{\mathbf{i}}
A selection of one element from each set can be
given by the vector $p \in \Int_+^m, p\leq \bar{n}$.
Let  $b_p = f(S) - f(\bigcup_{i=1}^m A_{-p_i}^i)$ be the unique contribution of the selected
elements to the total reward of $S$.
For convenience of notation, define $\veci := (i,i,\ldots,i) \in \real^m$.

The following lemma will help show the desired result.
\begin{lemma}
\label{lem:existsSmallSubset}
$\sum_{i=1}^k b_{\veci} \leq f(S)$.
\end{lemma}
\begin{proof}
The basic argument is that $b_{\veci}$ is the minimum possible contribution
of the set $\bigcup_j a_i^j$. So the total contribution over different $i$ will be
less than (or equal to) the sum of their minimum contributions.

Let $B_i =  \bigcup_{j=1}^m a_{i}^j$ be the set obtained by selecting the $i$th element from each set.

Using the definition of marginal benefit $\Delta_S(X)$,
we can write
$f(T) = \Delta_\emptyset(B_1) + \Delta_{B_1}(B_2) + \ldots + \Delta_{T \setminus B_k}(B_k)$.

However, by submodularity,
$b_{\veci} = \Delta_{S \setminus B_i}(B_i) \leq \Delta_X(B_i), \forall X \subseteq S\setminus B_i$.
Therefore, $f(T) \geq \sum_{i=1}^k \Delta_{S \setminus B_i}(B_i)$.

Combining this with the fact that $T \subseteq S \implies f(T) \leq f(S)$ by monotonicity, we get the desired result.
\end{proof}
To prove the theorem statement,
assume that there does not exist any set with the desired property, i.e.
$\forall p \in \Int_+^m, p\leq \bar{n}$ we have
$f(\bigcup_{i=1}^m A_{-p_i}^i) < (1 - \frac{1}{k}) f(S)$.
From this assumption, we can see that $b_p > \frac{1}{k}f(S)$.
Therefore, $\sum_{i=1}^k b_{\veci} > f(S)$.
With Lemma~\ref{lem:existsSmallSubset} the desired result is obtained by contradiction.

For the second part of the inequality, since $k = \min n_i \geq 3$,
so $(1-\frac{1}{k}) \geq \frac{2}{3}$.
\end{proof}

\begin{corollary}
\label{cor:existsSmallSubset}
Given a set $S$ such that $|S| \geq mk$, then 
there exists a subset $T$ of $m$ elements such that
$f(S\setminus T) \geq (1 - \frac{1}{k})f(S)$.
\end{corollary}
\begin{proof}
Since $|S| \geq mk$, $S$ can be divided into at least $m$ subsets
of size $\geq k$ each.
Creating such a division means that $S$ can be represented as
$\bigcup_{i=1}^m A^i$ such that $|A^i| \geq k, \forall i$.
Therefore, the result of Theorem~\ref{thm:existsSmallSubset} applies.
\end{proof}

\subsubsection{Algorithm to remove one element per set}
Based on the above results for existence of a set that can be removed
while maintaining at least $\frac{2}{3}$ of the original value,
we introduce a simple technique (given in Algorithm~\ref{alg:reduceSet})
to find such a set by searching over a finite number of disjoint sets.

\begin{algorithm2e}
\KwIn{$S = \bigcup_{i=1}^m A^i$
where $A^i = \bigcup_{j=1}^{n_i} a_j^i$}
\KwOut{$U \subset S$
s.t. $U\cap A^i = 1$ for all $i = 1 \ldots m$}
%\tcp{Notation: $\vec{i} = i \mathbf{1}_m$}
$i \leftarrow 1$;  $k \leftarrow \min_j |A^j|$\;
$U := \bigcup_{j=1}^m a^j_i$\;
\While{$f(S\setminus U) < \frac{k-1}{k}f(S)$}
{
	$i \leftarrow i + 1$\;
	$U := \bigcup_{j=1}^m a^j_i$\;
}
\Return $U$\;
\caption{\textsc{ReduceSet}$(S,A^1,\ldots,A^m)$}
\label{alg:reduceSet}
\end{algorithm2e}

\begin{theorem}
\label{thm:correctSmallSubset}
Algorithm~\ref{alg:reduceSet} is correct.
\end{theorem}
To prove correctness, it will help to establish the following result.
\begin{lemma}
\label{lem:correctSmallSubset}
Given a set $S$ composed of $m$ disjoint subsets as defined above.
Then there exists $i \in \set{1,\ldots,k}$ such that 
$f(\bigcup_{j=1}^m A^j_{-i}) \geq (1-\frac{1}{k}) f(S)$.
\end{lemma}
\begin{proof}
The proof follows a similar logic as the proof of Theorem~\ref{thm:existsSmallSubset}.

Assume that $\forall i \in \set{1,\ldots,k}$, 
$f(\bigcup_{j=1}^m A^j_{-i}) < \frac{k-1}{k} f(S)$.
So $\forall i \in \set{1,\ldots,k}$,
$f(S) - f(\bigcup_{j=1}^m A^j_{-i}) > \frac{1}{k} f(S)$.
Which implies
$\sum_{i=1}^k
f(S) - f(\bigcup_{j=1}^m A^j_{-i}) > \frac{k}{k} f(S) = f(S)$.
But we know that
$f(S) \geq \sum_{i=1}^k f(S) - f(\bigcup_{j=1}^m A^j_{-i})$
(by Lemma~\ref{lem:existsSmallSubset}).
Therefore, the desired result is shown by contradiction.
\end{proof}

\begin{proof}[Proof of Theorem~\ref{thm:correctSmallSubset}]
The algorithm (randomly) selects one element to remove from each set.
If the removal means more than $\frac{1}{k}$ of the reward is lost,
another set of elements is chosen that is disjoint from the previously chosen one.

Since each set $A^j$ contains atleast $k$ elements, there are $k$ possible
disjoint sets to choose from (for removal).
Lemma~\ref{lem:correctSmallSubset} states that one of these is guaranteed to have the desired property
of lowering the objective by at most $\frac{1}{k}$.
\end{proof}

\begin{theorem}
The complexity of Algorithm~\ref{alg:reduceSet} is $O(kf) = O(|S|f)$.
\end{theorem}

\subsubsection{Algorithm to delete edges}
We can now use Algorithm~\ref{alg:reduceSet} to remove one edge from each subtour in a matching.  

The subtours, $T^i$, can be considered as a collection of disjoint 
edge sets. Each subtour will consist of atleast three edges.
Since we want to remove one element from each set
while trying to maximize a submodular reward function,
the results of Theorem~\ref{thm:existsSmallSubset} apply and 
so we know that there will
exist a solution such that at most $\frac{1}{3}$ of the 
value of the objective is lost.

If there exists an extra edge not part of any subtour,
it does not affect the result since,
following from Corollary~\ref{cor:existsSmallSubset},
it can just be considered as part of one of the other subtours.
Since any $k$ of the elements of any
subtour are needed for the algorithm, these $k$ can be chosen to not include
the extra edge.
Outlined in Algorithm~\ref{alg:reduceMatching} is an method
that will find such a set of edges to remove.

\begin{algorithm2e}
\KwIn{A 2-matching $G_M = (V,M)$ where 
$M = \bigcup_{i=1}^m T^i$ and the sets $T^i$ are the subtours.}
\KwOut{$R \subset M$ such that each subtour has one edge removed.}
$n \leftarrow 0$\;
\For{$i = 1 \ldots m$}
{
	\If{$|T^i| > 1$}
	{
		$A^n = T^i$\;
		Increment $n$\;
	}
}
\Return  \textsc{reduceSet}$(M,A^1,\ldots,A^n)$\;
\caption{Reduce Matching}
\label{alg:reduceMatching}
\end{algorithm2e}

\begin{theorem}
\label{thm:reduceMatchingBound}
Algorithm~\ref{alg:reduceMatching} correctly reduces the matching while
maintaining $\frac{2}{3}$ of the original value.
The complexity of the algorithm is $O(kf) = O(|V|f)$
\end{theorem}

\subsection{Tour using matching algorithm}
\begin{figure}
\centering
\includegraphics[width=0.98\linewidth]{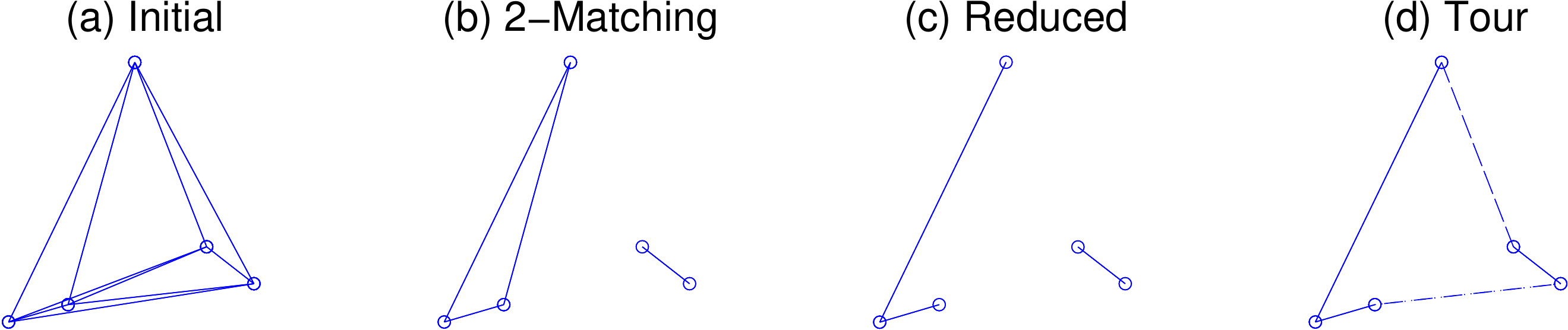}
\caption{The steps in the 2-matching based tour algorithm}
\label{fig:matchingSteps}
\end{figure}
We now present an outline of the complete 2-matching tour algorithm.
The steps are illustrated in Figure~\ref{fig:matchingSteps}.
\begin{enumerate}
\item Run Algorithm~\ref{alg:greedyMatching} to get a
simple 2-matching, $M_1$.
Using the linear relaxation $\tilde{w}$ of $w$, solve for the maximum
weight 2-matching, $M_2$.
From $M_1$ and $M_2$, choose the 2-matching that has a higher reward.
\item Find all sets of subtours.
(using, for example, DFS).
\item Run Algorithm~\ref{alg:reduceMatching} to select edges
to remove.
\item Connect up the reduced subtours into a tour.
One method to accomplish this would be by finding all the vertices that have
degree of 0 or 1 and then arbitrarily connecting them up to complete the tour.
\end{enumerate}

\begin{theorem}
\label{thm:matchingBoundRatio}
The 2-matching tour algorithm gives a 
$\max\set{\frac{2}{3(2+\kappa)},\frac{2}{3}(1-\kappa)}$-approximation
in $O(f(n^3+n) + n^3 \log n)$ time.
\end{theorem}
\begin{proof}
The matching is a $\max\set{\frac{1}{2+\kappa},1-\kappa}$-approximation from
Theorems~\ref{thm:greedyMatchingBounds} and \ref{thm:matchingLnrApprox}.
Given that a tour is a special case of the maximum simple 2-matching,
the maximum tour has a value less than or equal to the optimal 2-matching.
Removing edges from the matching produces a subgraph that is
atleast $\frac{2}{3}$ of the original value of the 2-matching 
(Theorem~\ref{thm:reduceMatchingBound}).
Adding more edges to complete the tour will only increase the value.
Therefore, the resulting tour is within
$\max\set{\frac{2}{3(2+\kappa)},\frac{2}{3}(1-\kappa)}$ of the optimal tour.
Note that the bound given is relative to the optimal 2-matching and not to 
the optimal tour so the actual bound may be better.

The first step involves building the matching. The greedy matching takes 
$O(|V|^3(f + \log |V|))$ time and the linear relaxation approximation
takes $O(|V|^3)$ time.
Finding the edges in all the components is
$O(|V|+|E|) = O(|V|)$ since the number of edges in a matching
is at most $|V|$.
Removing the edges is $O(|V|f)$.
Assuming it is easy to find the free vertices,
connecting up the final graph is $O(m) = O(|V|)$ where $m$
is the number of subtours and ranges from $1$ to $\floor{\frac{|V|}{3}}$.
Therefore, the total runtime is
$O(n^3 (f + \log n) + n(2+f)) = O(f(n^3+n) + n^3 \log n)$.
\end{proof}

The reason for calculating the 2-matching twice is now explained.
The greedy method gives a $\frac{1}{2}$-approximation
in the case of a linear objective (Lemma~\ref{lem:greedylnrpsystem}).
A similar method of using a matching is used in \cite{MLF-GLN-LAW:79} for a linear reward function.
In that case, the approximation ratio achieved is $\frac{2}{3}$ of the optimal tour.
The reason for this
is that for a linear function, an optimal perfect matching is obtained
and the bound depends only on how much is lost in removing the edges.
We showed that a similar bound limiting the loss can be obtained for the submodular case; 
however, just using the greedy approximation for the 2-matching,
we cannot guarantee as good a result since the final tour would be within
$\frac{2}{3}\frac{1}{2}=\frac{1}{3}$  of the optimal.
By also using the second method of finding the 2-matching, our resulting 
bound for the final tour in the case of a linear function improves to $\frac{2}{3}$.
\begin{remark}
For any value of $\kappa < \frac{1}{2}(\sqrt{3}-1) \approx 0.366$,
using the linear relaxation method
to construct a 2-matching and then converting it into a tour, gives a better bound
with respect to the optimal tour than by using the greedy tour approach.
\oprocend
\end{remark}

\begin{remark}
In our case, the $\frac{2}{3}$ loss is actually a worst case bound where all the subtours are
composed of 3 edges.
For a graph with a large number of vertices, the subtours will probably be larger and so 
$k$ may be larger than 3. This would lead to a better bound for the algorithm.
\oprocend
\end{remark}
\begin{remark}
In removing the edges we used an algorithm to quickly find a ``good" set of edges to
remove but made no effort to look for the ``best" set. Using a more intelligent heuristic we would
get better results (of course at the cost of a longer runtime).

Also, the last step where the reduced subtours are connected into a tour can be achieved using various
different techniques. As mentioned above, one possibility is to just arbitrarily
 connect up the components.
Another method would be to use a greedy approach (so running Algorithm~\ref{alg:greedyTour} except
with an initial state).
This would not change the worst case runtime 
given in Theorem~\ref{thm:matchingBoundRatio}.
Alternatively, if the number of subtours is found to be small, an exhaustive search 
going through all the possibilities could be performed.
\oprocend
\end{remark}

\section{Extension to directed graphs}
\label{sec:directedExtension}
The algorithms described can also be applied to directed graphs yielding
approximations for the ATSP.

\subsection{Greedy Tour}
For the greedy tour algorithm, a slight modification needs to be made to check
that the in-degree and out-degree of the vertices are less than $1$ instead of
checking for the degree being less than $2$.
Since the ATSP is a 3-extendible system, the approximation of the greedy algorithm
changes to $\frac{1}{3+\kappa}$ instead of $\frac{1}{2+\kappa}$ as in the undirected case.

\subsection{Tour using Matching}
Instead of working with a 2-matching, the system can be modelled as the intersection
of two partition matroids:
\begin{itemize}
\item Edge sets such that the indegree of each vertex $\leq 1$.
\item Edge sets such that the outdegree of each vertex $\leq 1$.
\end{itemize}
This system is still 2-extendible and so the approximation for the greedy 2-matching
does not change.
For the second approximation, the Maximum Assignment Problem (max AP) can be solved
by representing the weights in \eqref{eqn:lnrRelaxWeights} as a weight matrix
$\tilde{W}$ where we set $\tilde{W}_{ii} = -\infty$.
The Hungarian algorithm, that has a complexity of $O(n^3)$, can be applied
to obtain an optimal solution. Therefore, the result of Theorem~\ref{thm:matchingLnrApprox}
still applies.

The result of the greedy algorithm or the solution to the assignment problem 
will be a set of edges that together form
a set of cycles, with the possibility of a lone vertex.
Note that a ``cycle" could potentially consist of just two vertices.
Therefore, removing one edge from each cycle will result in a loss of at most
$\frac{1}{2}$ instead of $\frac{1}{3}$. This follows directly from the analysis
in Theorem~\ref{thm:existsSmallSubset} but using $k=2$ instead of $k=3$.
The final bound for the algorithm is therefore
$\max\set{\frac{1}{2(2+\kappa)},\frac{1}{2}(1-\kappa)}$.

\section{Incorporating Costs}
\label{sec:incCosts}
Often times, optimization algorithms have to deal with multiple objectives.
In our case, we can consider the tradeoff between the reward of a set and its associated cost.
A number of algorithms presented in literature look at attempting to maximize the benefit given a ``budget" or a bound on the cost,
i.e. find a tour T such that
\[
T \in \argmax_{S \in \mathcal{H}} w(S) \text{ s.t. } c(S) \leq k.
\]
This involves maximizing a monotone non-decreasing submodular function over
a knapsack constraint as well as an independence system constraint.

We will work with a different form of the objective function 
defined by a weighted combination of the reward and cost.
For a given value of $\beta \in[0,1]$, solve
\begin{align}
T &\in \argmax_{S \in \mathcal{H}} f(S) \\
f(S,\beta) &= (1-\beta) w(S) - \beta c(S), \label{eqn:combObj1}
\end{align}
where the combined objective is a non-monotone possibly negative submodular function.
An advantage of having this 
form for the objective is that the ``cost trade-off" is being incorporated directly into 
the value being optimized.  Since the cost function is modular, maximizing the negative of 
the cost is equivalent to minimizing the cost.  So the combined objective tries to 
maximise the reward and minimize the cost at the same time. 

The parameter $\beta$ is used as a weighting mechanism.
The case of $\beta = 0$ corresponds to ignoring costs and
that of $\beta=1$ corresponds to ignoring rewards and just minimizing the cost
(this would just be the traditional TSP).

One minor issue with the proposed function is that the rewards and costs may have
different relative magnitudes.
This might mean that $\beta$ is biased to either 0 or 1.
Normalizing the values of $w$ and $c$ will help bring both the values down
to a similar scale.
This gives the advantage of being able to use $\beta$ as an unbiased tuning parameter.

Therefore, the final definition of the objective function is
\begin{align}
f(S) = \frac{1-\beta}{M_w}w(S) - \frac{\beta}{M_c}c(S)
\label{eqn:combObj}
\end{align}
where $M_c = \max_{S}{c(S)}$ and $M_w = \max_{S}{w(S)}$.
The exact values of $M_c$ and $M_w$ may be hard to calculate
and so could be approximated.

To address the issue of the function being non-monotone and negative,
consider the alternative modified cost function
\[
c'(S) = |S|M - c(S), \quad M = \max_{e \in E} c(e).
\]
This gives the following form for the objective function,
\begin{align}
f'(S,\beta) &= (1-\beta) w(S) + \beta c'(S) \label{eqn:combObj2} \\
&= (1-\beta) w(S) + \beta (|S|M - c(S)), \nonumber \\
&= f(S,\beta) + \beta|S|M, \nonumber
\end{align}
which is a monotone non-decreasing non-negative submodular function.
This has the advantage of offering known approximation bounds.
The costs have in a sense been ``inverted" and so maximizing $c'$ still corresponds to minimizing 
the cost $c$.
\begin{remark}
Instead of using $|S|M$ as the offset, we could use the sum of the $|S|$ 
largest costs in $E$. This would not improve the worst case bound but
may help to improve results in practice.
\oprocend
\end{remark}

\begin{lemma}
\label{lem:newFuncOrder}
For any two sets $S_1$ and $S_2$,
if $f'(S_1) \geq \alpha f'(S_2)$ then $f(S_1) \geq \alpha f(S_2) + M( \alpha|S_2| - |S_1|)$
\end{lemma}
\begin{proof}
\begin{align*}
&f'(S_1) \geq \alpha f'(S_2)
\\
\implies
&f(S_1) +|S_1|M \geq \alpha (f(S_2) +|S_2|M)
\\
\implies
&f(S_1) \geq \alpha f(S_2) + M ( \alpha|S_2| -|S_1| )
\end{align*}
\end{proof}

\begin{remark}
\label{rem:newFuncOrder}
As a special case, if $|S_1| = |S_2|$,
then $f(S_1) > f(S_2)$ if and only if $f'(S_1) > f'(S_2)$.
Therefore, it can be deduced that over all sets of the same size,
the one that maximizes (\ref{eqn:combObj1}) is the same one that maximizes (\ref{eqn:combObj2}).
\oprocend
\end{remark}

\begin{theorem}
\label{thm:combObjResult}
Let $G_1$ be the greedy solution obtained maximizing \eqref{eqn:combObj1}
and $G_2$ be the greedy solution maximizing \eqref{eqn:combObj2}.
Then $G_1 = G_2$.
\end{theorem}
\begin{proof}
At each iteration $i$ of the greedy algorithm, we are finding the element that will
give the maximum value for a set of size $i+1$.
Since comparison is being done between sets of the same size,
the same element will be chosen at each iteration.
\end{proof}

\begin{theorem}
\label{thm:approxBoundwCost}
Consider a function $f = f_1 - f_2$, where $f_1$ is submodular and $f_2$ is modular,
and a $p$-system $(E,{\cal F})$.
An $\alpha$-approximation to the problem $\max_{S \in {\cal F}} f'(S)$,
where $f'(S) = f(S) + |S|M, M = \max_{e\in E} f_2(e)$,
corresponds to an approximation of $\alpha OPT - \big( 1 + \alpha - 2\frac{\alpha}{p} \big)Mn$, for the problem $\max_{S \in {\cal F}} f(S)$,
where $n$ is the size of the maximum cardinality basis.
\end{theorem}
\begin{proof}
Let $S$ be the solution obtained by a $\alpha$-approximation algorithm to $f'$.
Let $T$ be the optimal solution using $f'$.
Let $Z$ be the optimal solution using $f$.
Note the following inequality:
\begin{align*}
|A| \leq |B|p
\implies 
\alpha|B| -|A| \geq |B|(\alpha - p) \geq |A|(\frac{\alpha}{p} - 1).
\end{align*}
By using the property of $p$-systems that for any two bases $A$ and $B$,
\[
\frac{|A|}{|B|} \leq p,
\]
and by Lemma~\ref{lem:newFuncOrder}, we have
\[
f'(S) \geq \alpha f'(T) \implies f(S) \geq \alpha f(T) + M\left(\alpha |T| -|S|\right).
\]
So,
\begin{align*}
f(S) &\geq \alpha f(T) - M|T| \left( p - \alpha \right)
\\
&\geq \alpha f(T) - M|S| \left( 1-\frac{\alpha}{p} \right)
\\
&\geq \alpha f(T) - Mn \left( 1-\frac{\alpha}{p} \right).
\end{align*}
Also,
\[
f'(T) \geq f'(Z) \implies f(T) \geq f(Z) - M|T|\left( 1 - \frac{1}{p}\right).
\]
Substituting,
\begin{align*}
f(S)
&\geq \alpha f(Z) - \alpha M|T|\left( 1 - \frac{1}{p}\right) - Mn \left( 1 - \frac{\alpha}{p} \right)
\\
&\geq \alpha f(Z) - Mn \left( 1 + \frac{p-2}{p}\alpha \right).
\end{align*}
\end{proof}

\begin{remark}
In the special case of a 1-system (this includes matroids), or more generally
any problem where the output to the algorithm will always be the same size,
we have $|S| = |T| = |Z|$ and also $T = Z$ following from Remark~\ref{rem:newFuncOrder}.
This means that an algorithm that gives a relative error of $\alpha$ when using $f'$
as the objective will give a normalized relative error of $\alpha$ when using
$f$ as the objective
(i.e. $f(S) \geq \alpha {\rm OPT} + (1-\alpha) {\rm WORST}$).
\oprocend
\end{remark}

\begin{corollary}
The problem $\max_{S \in {\cal F}} f(S)$
can be approximated to $(p+1)^{-1}OPT + (1+\epsilon)WORST$
(where $\epsilon \in [-\frac{1}{2},1)$ is a function of $p$)
using a greedy algorithm.
\end{corollary}
\begin{proof}
Let $f'(S) = f(S) + |S|M, M = \max_{e \in E}f_2(e)$ and run the greedy algorithm
to obtain the set $T_G$.
Let Let $T$ and $Z$ be defined as above.
Since $f'$ is a non-negative monotone function,
\[ f'(T_G) \geq \frac{1}{p+1}f'(T). \]
So applying Theorem~\ref{thm:approxBoundwCost},
\begin{align*}
f(T_G) &\geq \frac{1}{p+1}f(Z) - Mn \left( 1 + \frac{p-2}{p(p+1)} \right)
\end{align*}
\end{proof}

\subsection{Performance Bounds with Costs}
Using the proposed modification, new bounds can be derived for the algorithms
discussed in this paper.
One thing to note is that in the case of the tour, all tours will be the same
length that is $|V|$, even though the tour is not a 1-system.
Therefore, we can apply Lemma~\ref{lem:newFuncOrder} directly.

\begin{theorem}[Greedy tour with edge costs]
Using \eqref{eqn:combObj1} as the objective, Algorithm~\ref{alg:greedyTour}
outputs a tour that has a value at least
$ \frac{1}{3}OPT - \frac{2}{3}\beta M|V| $
where $M$ is the maximum cost of any edge.
\end{theorem}

\begin{theorem}[Tour via 2-matching with edge costs]
Using \eqref{eqn:combObj1} as the objective, the tour based on a matching algorithm
outputs a tour that has a value at least
$ \frac{2}{9}OPT - \frac{7}{9}\beta M|V| $
where $M$ is the maximum cost of any edge.
\end{theorem}

The bounds given in these theorems are not the best that can be obtained.
For the case when $\beta$ is small, the submodular reward is weighted higher and the bound
is closer to that of maximizing a submodular function.
On the other hand, for the case of large $\beta$ (specifically when $\beta = 1$), the problem
is just the traditional minimum cost TSP. In \cite{TAJ:79} an approximation of
$\frac{1}{2}(OPT+WORST)$ was given for finding a minimum TSP using a greedy approach which
is better than the $\frac{1}{3}OPT+\frac{2}{3}WORST$ that we calculate for the greedy tour with costs.
In addition, simple fast methods also exist to find the minimum cost tour in a graph using
other approaches.
Therefore, if the costs are to be given a higher weight, the analysis given here is not very
informative of the resulting solution.

\section{Simulations}
\label{sec:sims}

\begin{figure}[tb]
\centering
\subfloat[Example graph.]{
	\includegraphics[width=0.40\linewidth]{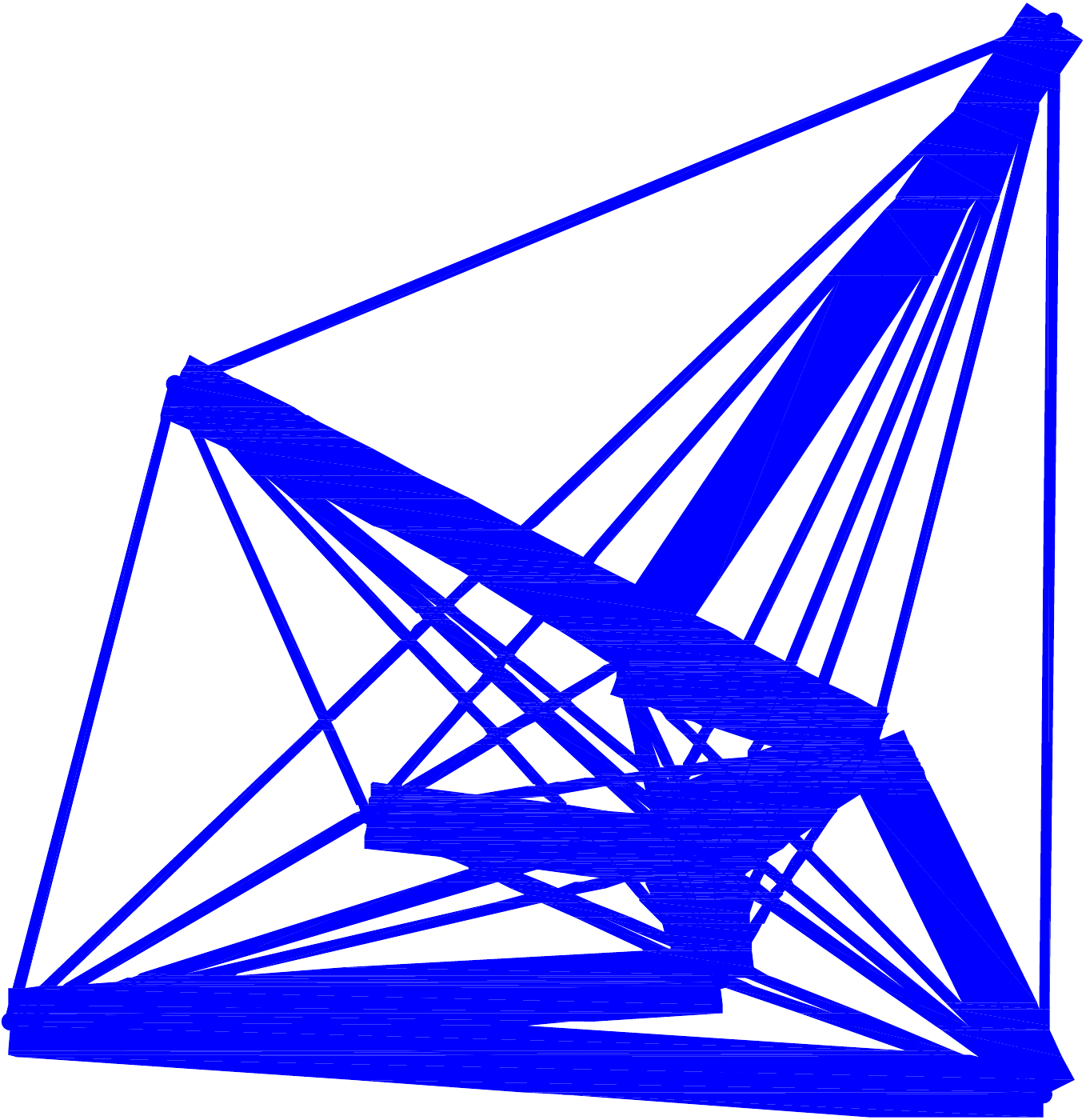}
	\label{fig:egGraph}
}
\hfill
\subfloat[Greedy solution.]{
    \includegraphics[width=0.40\linewidth]{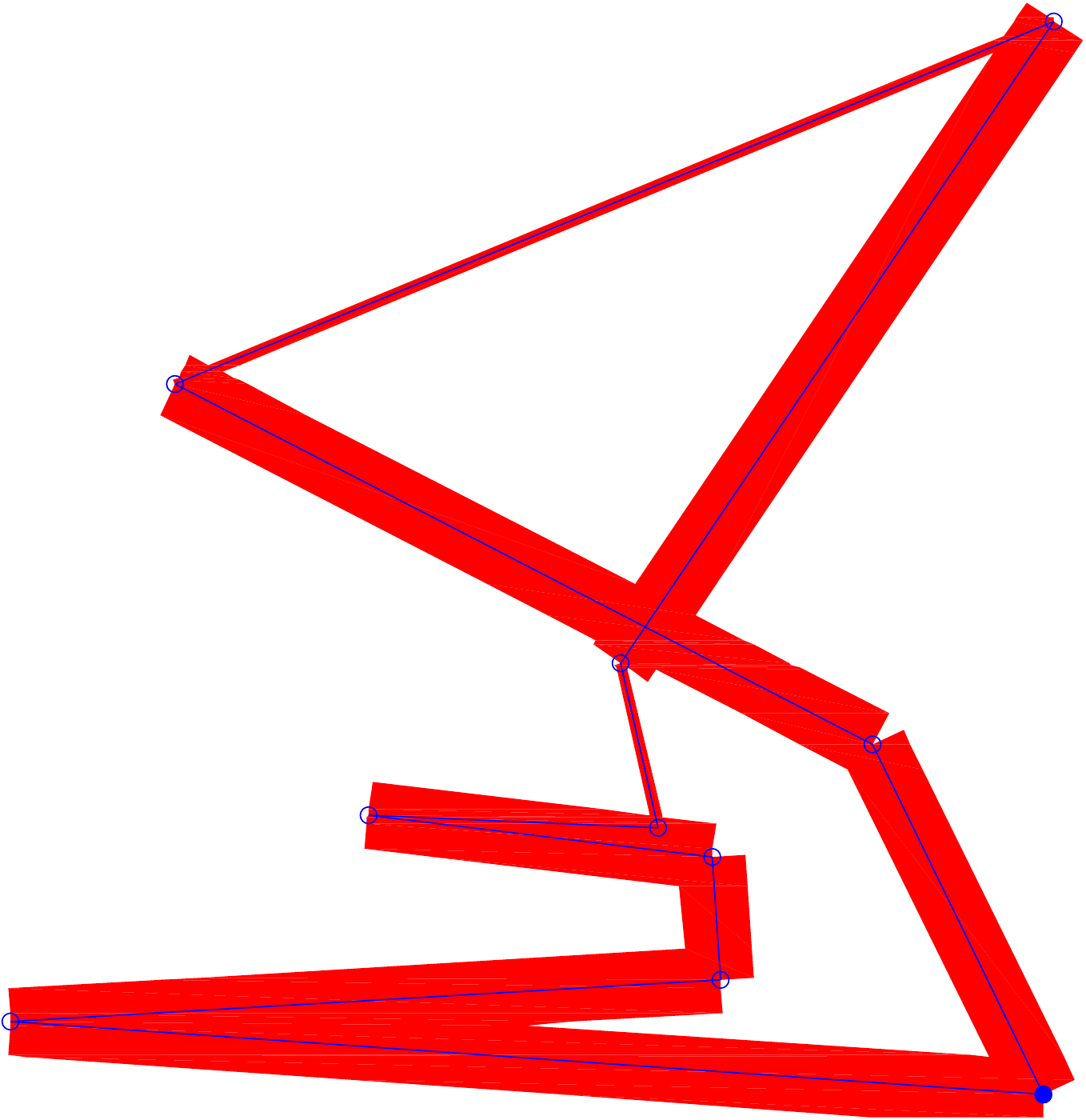}
    \label{fig:egGraphGreedy}
}
%\hfill
%\subfloat[Random solution.]{
%    \includegraphics[width=0.31\linewidth]{simpleGraphRandom.pdf}
%    \label{fig:egGraphRandom}
%}
\caption{A ten vertex graph and example solution.}
\label{fig:egGraph2}
\end{figure}

In order to empirically compare our algorithms, we have run simulations for a function that 
represents coverage of an environment.  A complete graph is generated by uniformly placing 
vertices over a rectangular region.  Each edge in the graph is associated with a rectangle 
and each rectangle is assigned a width to represent different amounts of coverage.

An example of a ten vertex graph is given in Figure~\ref{fig:egGraph}.
Here we see the complete graph as well as a representation of the value of each
edge given by the area of the rectangle the edge corresponds to.
The majority of edges have a low weight with a few having a much larger value.
Running the greedy tour algorithm, we get the tour given in
Figure~\ref{fig:egGraphGreedy}.

The simulations were performed on a quad-core machine with a 3.10 GHz CPU and 6GB RAM.
To decrease total runtime, three instances of problems were run in parallel on
different Matlab\textsuperscript{\textregistered} sessions.

\subsection{Algorithm Comparison}
Next we compare the performance of the algorithms given in this paper.
Each algorithm was run on randomly generated graphs for a fixed number of
vertices. The resulting value of the objective function was recorded and
averaged over all instances.
The algorithms compared are:
\begin{itemize}
\item GreedyTour (GT): The greedy algorithm for constructing a tour.
\item RandomTour (RT): Edges are considered in a random order. An edge is selected to be part of the tour
as long as the degree constraints will be satisfied and no subtour will be created.
\item For the 2-matching based algorithm, three possibilities are considered.
All three start off by greedily constructing a 2-matching.
	\begin{itemize}
	\item GreedyMatching (GM): Remove from each subtour the element that will
	result in the least loss to the total value.
	Greedily connect up the complete tour.
	\item GreedyMatching2 (GM2): Use Algorithm~\ref{alg:reduceMatching} to reduce
	the matching. Greedily connect up the complete tour.
	\item GreedyMatching3 (GM3): Use Algorithm~\ref{alg:reduceMatching} to reduce
	the matching. Arbitrarily connect up the complete tour.
	\end{itemize}
\end{itemize}

The vertices were distributed randomly over a 100x100 region.
For the first simulation, the edge
thickness was assigned a value of 7 with probability $\frac{2}{\sqrt{|V|}}$,
or 1 otherwise.
This way, $O(|V|)$ of the edges had a high reward. A total of 30 different
instances of the problem were solved by all the algorithms for 5 different
graph sizes.
For the second simulation, the set up was the same except the edges thickness
were distributed uniformly over $[0,7]$ and a total of 40 instances were averaged.

\begin{figure}
\centering
\includegraphics[scale=0.4]{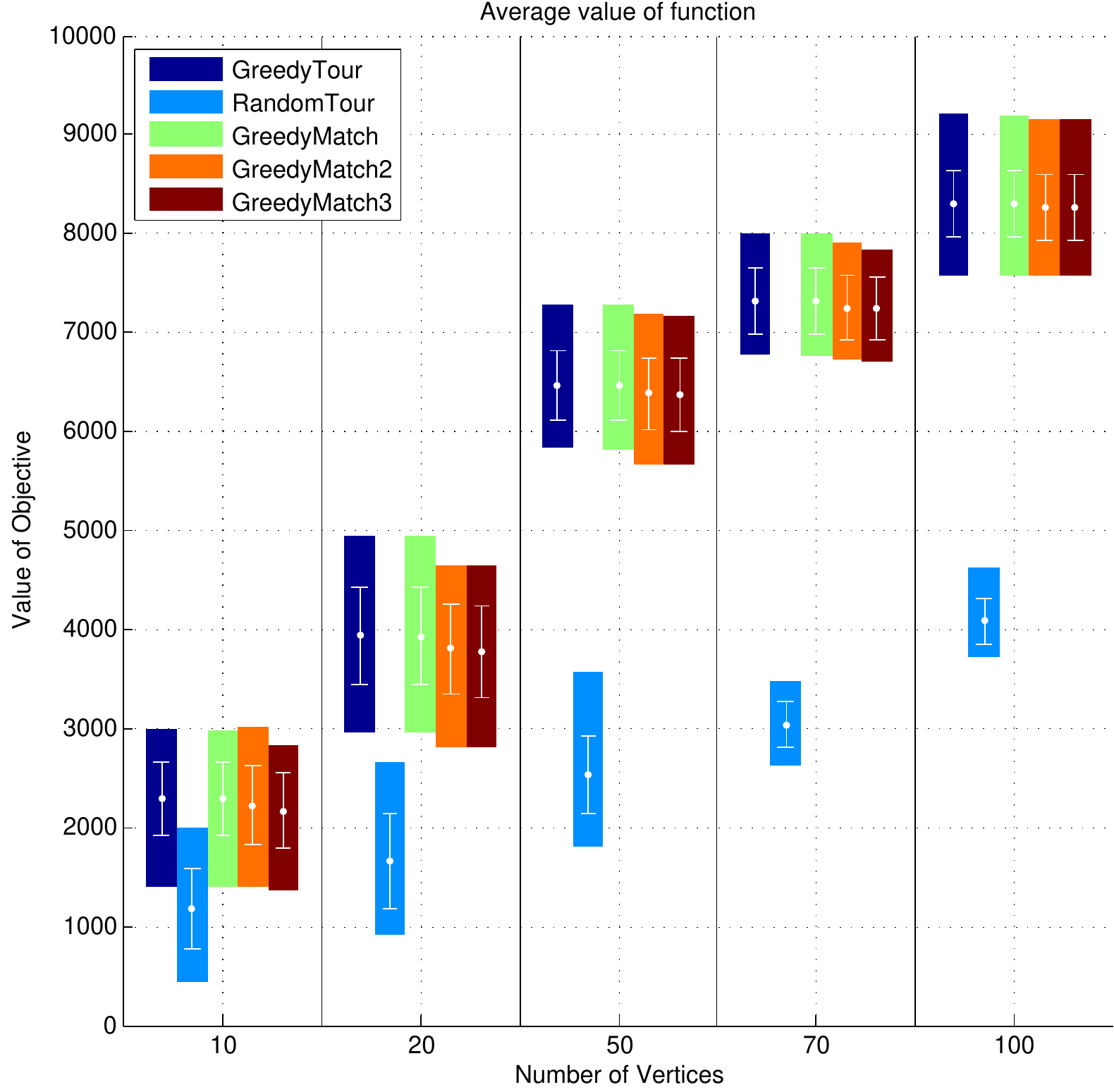}\\[0.5em]
{\footnotesize
\begin{tabular}{c||c|c|c|c|c}
 & GT & RT & GM & GM2 & GM3\\ \hline \hline
10 & 27 (3) & 0 (0) & 25 (1) & 16 (1) & 2 (0)\\ \hline
20 & 23 (8) & 0 (0) & 21 (5) & 8 (0) & 1 (0)\\ \hline
50 & 26 (16) & 0 (0) & 14 (4) & 5 (0) & 0 (0)\\ \hline
70 & 27 (18) & 0 (0) & 12 (3) & 3 (0) & 1 (0)\\ \hline
100 & 27 (16) & 0 (0) & 14 (3) & 2 (0) & 1 (0)\\
\end{tabular}
}
\caption{(Top) The bars give the range of results.
The white markers inside the bars show the mean and standard deviation.
(Bottom) Number of wins for each algorithm.
Wins include ties (unique wins specified in parens).}
%\caption{The bars give the range of results.
%The white markers inside the bars show the mean and standard deviation.}
\label{fig:simVals}
\end{figure}

\begin{figure}
\centering
\includegraphics[scale=0.4]{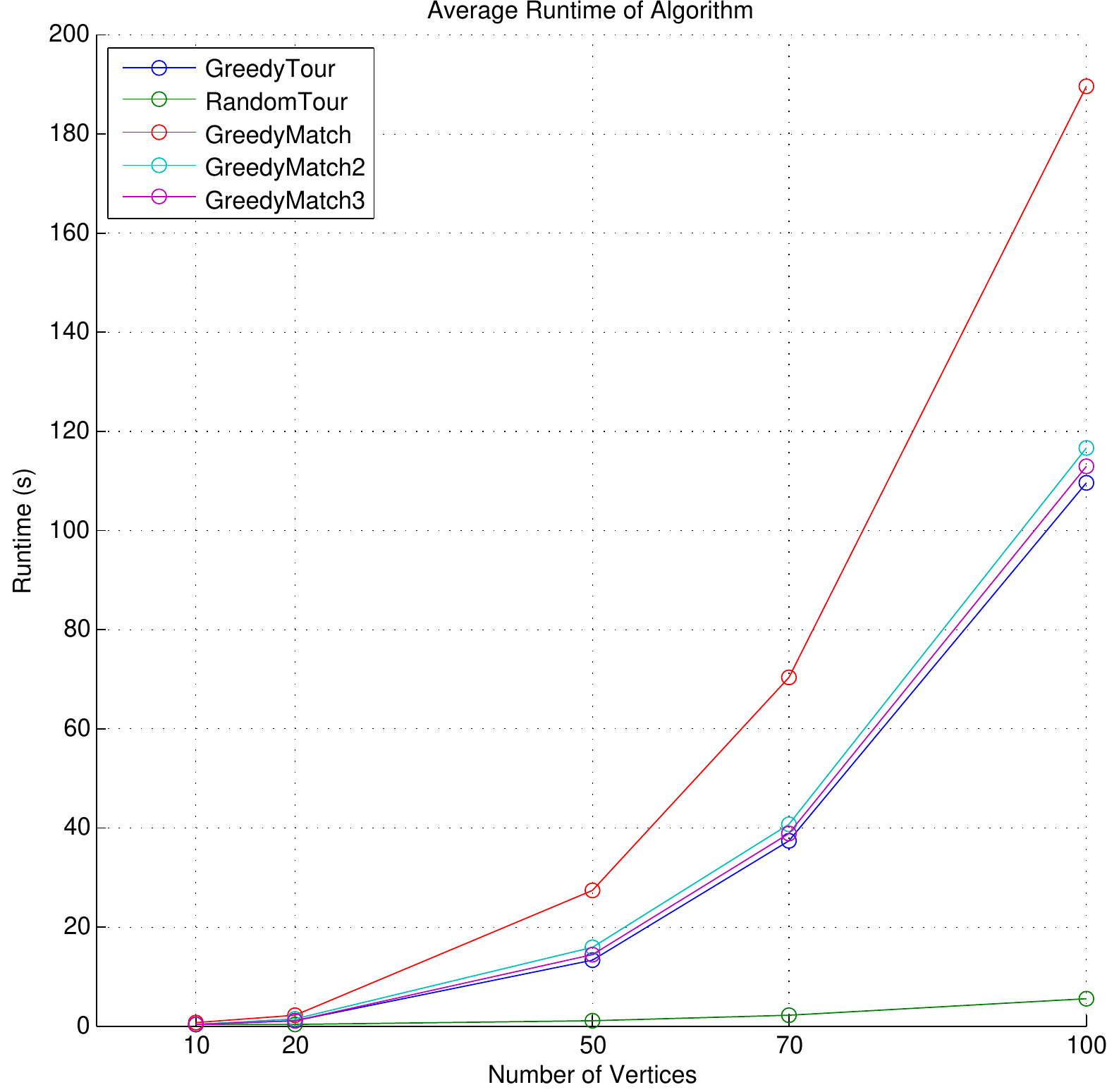}\\[0.5em]
{\footnotesize
\begin{tabular}{{c||c|c|c|c|c}}
 & GT & RT & GM & GM2 & GM3\\ \hline \hline
10 & 0.2 & 0 & 0.4 & 0.3 & 0.2\\ \hline
20 & 0.8 & 0.1 & 2.2 & 1.4 & 1.1\\ \hline
50 & 13.3 & 0.8 & 27.1 & 15.9 & 14.4\\ \hline
70 & 37.4 & 2.2 & 70.4 & 40.7 & 39.0\\ \hline
100 & 109.7 & 5.4 & 189.9 & 116.5 & 113.1\\
%10 & 0 & 0 & 0 & 0 & 0\\ \hline
%20 & 1 & 0 & 2 & 1 & 1\\ \hline
%50 & 13 & 1 & 27 & 16 & 14\\ \hline
%70 & 37 & 2 & 70 & 41 & 39\\ \hline
%100 & 110 & 5 & 190 & 117 & 113\\
\end{tabular}
}
\caption{Average runtime (seconds) of algorithms.}
\label{fig:simTimes}
\end{figure}

Average runtimes of each algorithm are shown in Figure~\ref{fig:simTimes}.

The results for the first setup are shown in Figure~\ref{fig:simVals}
The table gives a count of the number of times each algorithm had the largest value.
Average runtimes of each algorithm are shown in Figure~\ref{fig:simTimes}.
For the second setup, the solution values and number of wins are given in
Figure~\ref{fig:simVals18}.

The RT algorithm performs poorly in each case. This is expected as
no effort is put into finding good edges. For the first setup,
most of the edges have a low reward and so the rewards of any random set of edges
will be biased towards a small value.
In the second setup, since the distribution is uniform, the expected value of the
tour increases though getting close to the ``best" tour is still not likely.

Both GT and GM perform similarly well on average (note that the number of ties
is high  especially for smaller graph sizes) though GM takes a lot longer to run.
This can be explained due to the extra oracle calls required to
determine which edge to remove from each subtour
(a total of $|V|$ extra oracle calls).

Both GM2 and GM3 are slightly behind GM in terms of the final value.
This makes sense
since in GM more effort is put into finding a good reduction of the 2-matching 
whereas in GM2 only an attempt to find a good set of edges to remove is made.
Between GM2 and GM3, the solution values are very close though GM2 runs slightly 
slower than GM3. The only difference between the two techniques is that the
joining of the reduced subtours is performed randomly for GM3. This requires no
oracle calls leading to a faster runtime. In this particular set up, the number
of subtours was $o(|V|)$ so very few calculations were needed to construct the final
tour from the reduced subtours. It is however possible for the number of subtours
to be $O(|V|)$ and in those cases GM2 would be much slower as the problem size would
not be significantly reduced by first coming up with a 2-matching.

\begin{figure}
\centering
\includegraphics[scale=0.4]{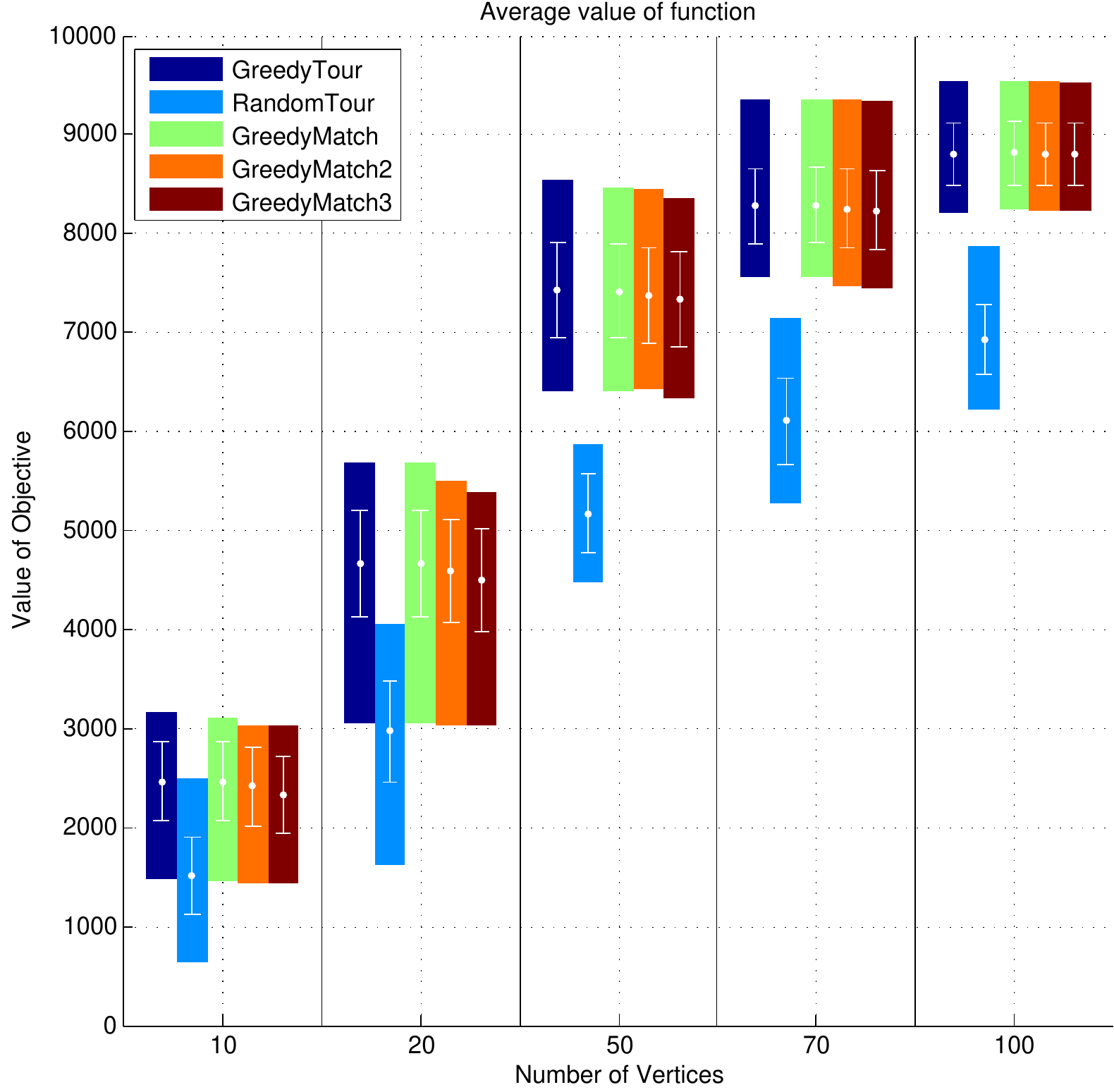}\\[0.5em]
{\footnotesize
\begin{tabular}{c||c|c|c|c|c}
 & GT & RT & GM & GM2 & GM3\\ \hline \hline
10 & 35 (10) & 0 (0) & 28 (3) & 22 (0) & 4 (0) \\ \hline
20 & 31 (8) & 0 (0) & 32 (9) & 11 (0) & 0 (0)\\ \hline
50 & 33 (13) & 0 (0) & 25 (5) & 12 (2) & 1 (0)\\ \hline
70 & 29 (10) & 0 (0) & 29 (10) & 8 (1) & 4 (0)\\ \hline
100 & 35 (5) & 0 (0) & 35 (5) & 17 (0) & 12 (0)\\
\end{tabular}
}
\caption{(Top) The bars give the range of results.
The white markers inside the bars show the mean and standard deviation.
(Bottom) Number of wins for each algorithm.
Wins include ties (unique wins specified in parens).}
\label{fig:simVals18}
\end{figure}

\subsection{Dependence on Curvature}
To illustrate how the results of the 2-matching based algorithm
changes with curvature, the values of the greedy matching and the linear
approximation are now compared.
The submodular function is modified to be
\[
f_{new}(S) = f(S) + \sum_{e \in S} length(e).
\]
Since $f(S)$ is the total area, its value depends on the thickness of the edges.
For a small thickness, the overlaps in the area between different edges
will be small and so the function will be more linear.
Therefore, there is a positive correlation between the edge thickness and the
curvature of the function.

The experiment is performed on a ten vertex graph.
The two 2-matching approximations are compared and the
results are shown in Figure~\ref{fig:curvCompareVals}.
For a second test, the GreedyTour and the 2-matching algorithm are compared and the
average for 20 different ten vertex graphs is shown in Figure~\ref{fig:curvCompareVals2}.
In the figure, the ``LGmatching" algorithm creates a greedy 2-matching as well as 
the linear approximation and takes the best of the two. The best edge from each subtour is
removed and the tour is then constructed greedily.
The ``Lmatching" algorithm runs only the linear approximation to find the 2-matching.
Figure~\ref{fig:curvCompareCurv} shows how the curvature changes with edge thickness.

\begin{figure}
\centering
\includegraphics[width=0.85\linewidth]{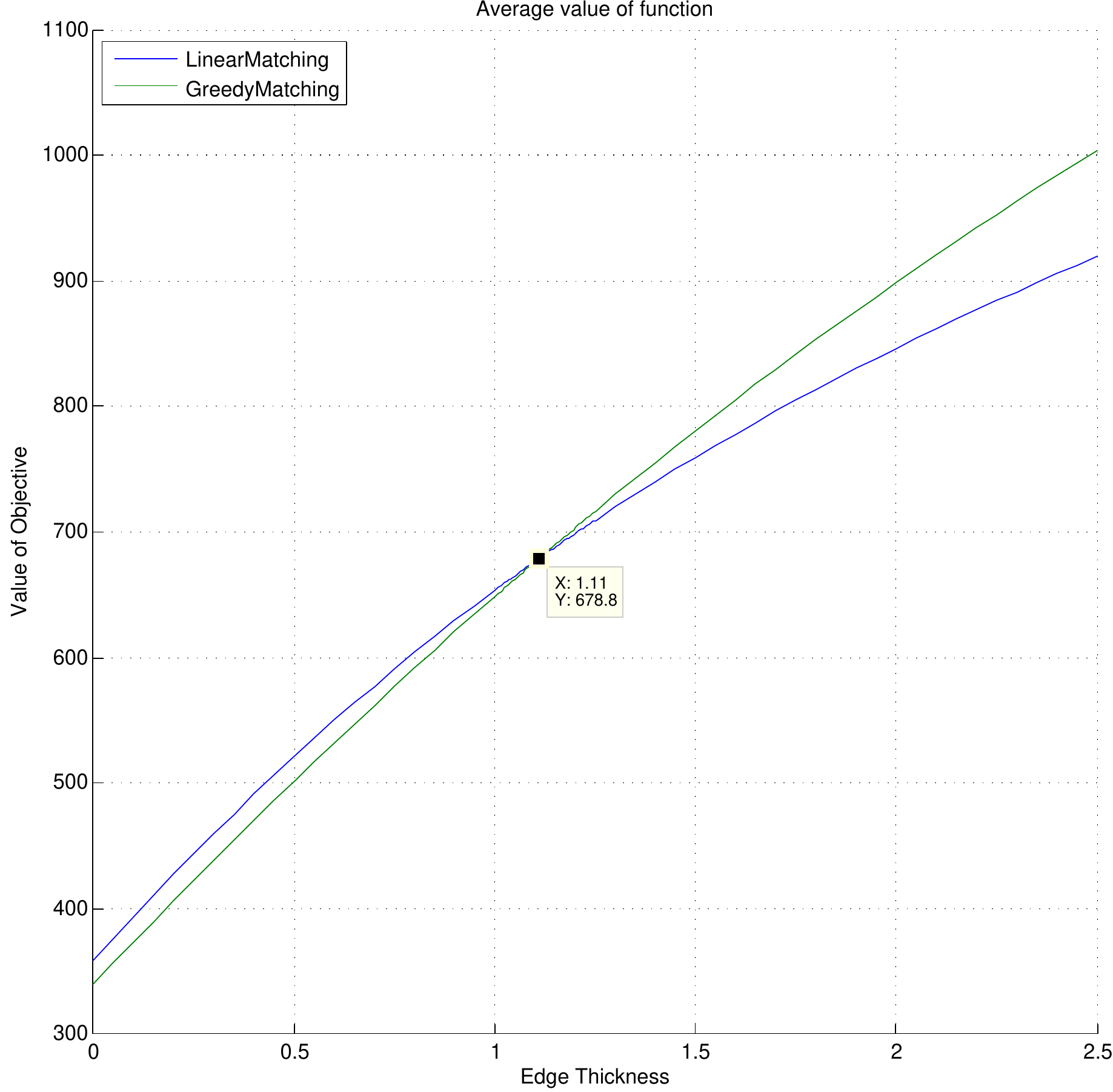}
\caption{Curvature comparison of 2-matching algorithms}
\label{fig:curvCompareVals}
\end{figure}

\begin{figure}
\centering
\includegraphics[width=0.85\linewidth]{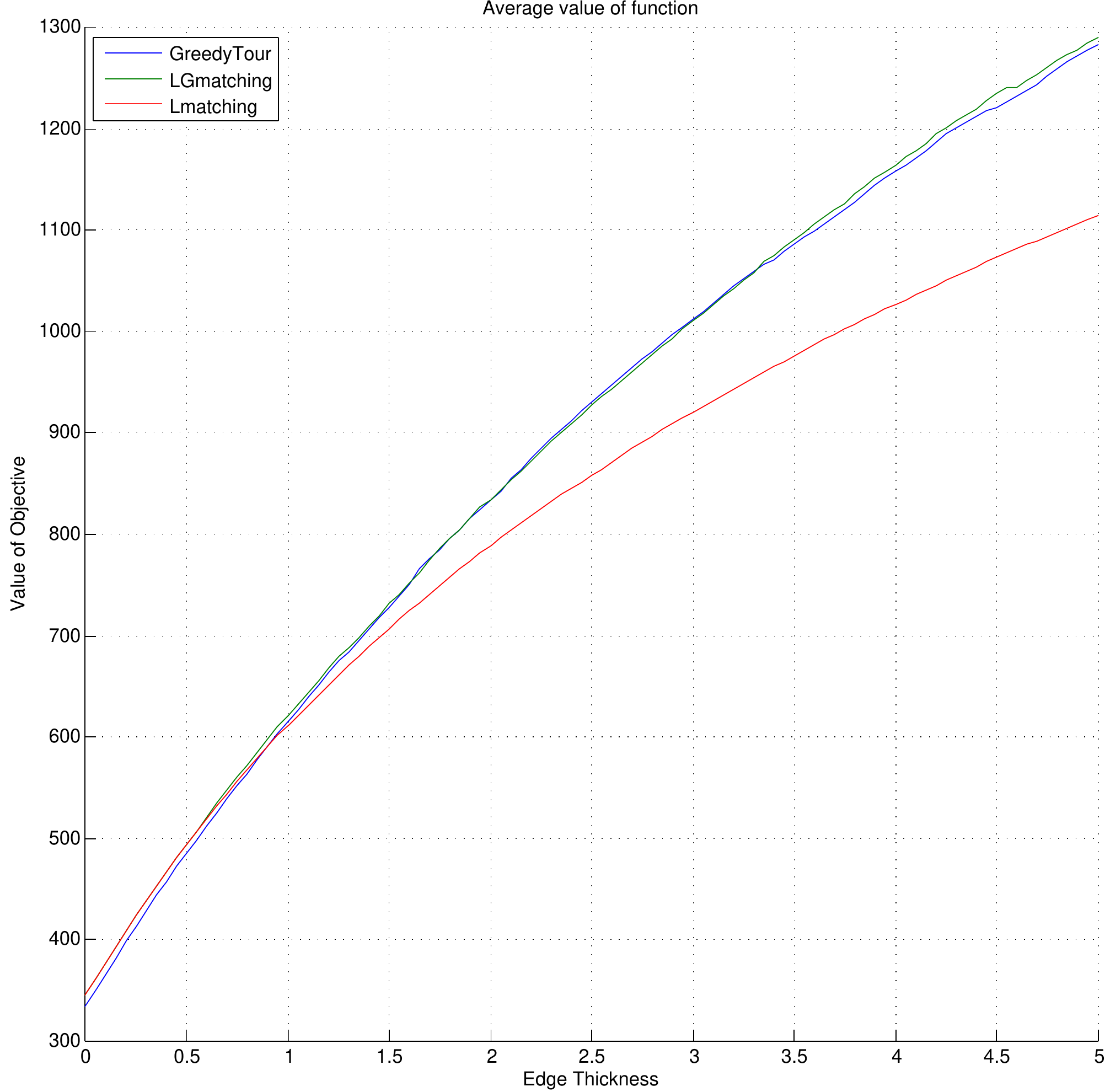}
\caption{Curvature comparison of final tour}
\label{fig:curvCompareVals2}
\end{figure}

\begin{figure}
\centering
\includegraphics[width=0.85\linewidth]{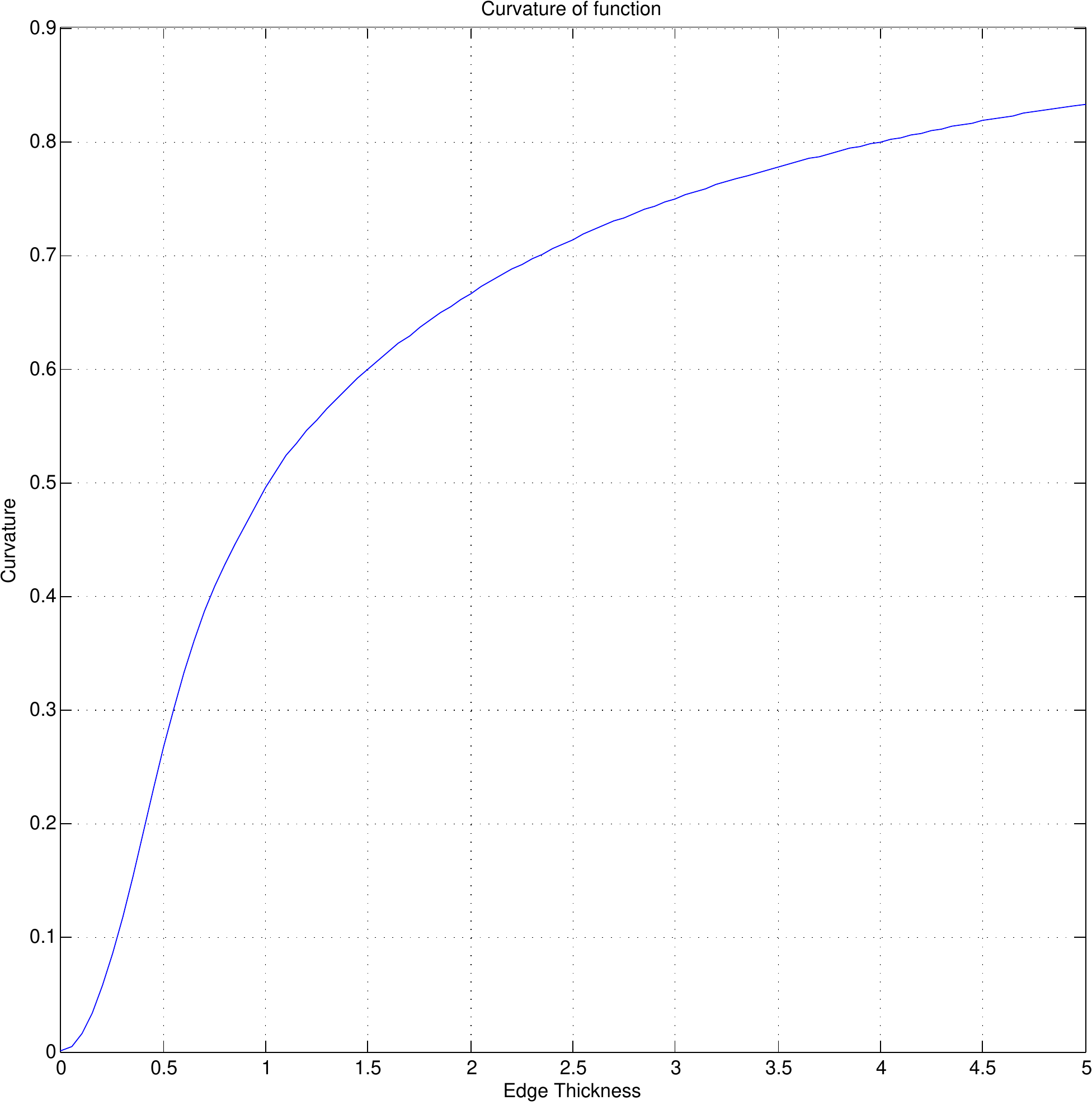}
\caption{Curvature as a function of edge thickness}
\label{fig:curvCompareCurv}
\end{figure}

From Figure~\ref{fig:curvCompareVals}, we can see that for the case where the function
is completely linear
(thickness is 0), the linear approximation does better (since it is
actually finding the optimal).
The greedy matching starts to perform better at a curvature of around 0.53.
Looking at the results for the value of the actual tour (Figure~\ref{fig:curvCompareVals2}),
we can see that at low values of curvature the linear approximation is being used
to create the tour. Eventually, greedily constructing the 2-matching becomes more
rewarding and so the linear approximation is disregarded.
Generally, over all the values of curvature tested, the matching algorithm
performs close to or better that the greedy tour algorithm.

\section{Conclusions and Future Directions}
\label{sec:conc}
In this paper, we extended the max-TSP problem to submodular rewards.
We presented two algorithms; a greedy algorithm which achieves a
$\frac{1}{2+\kappa}$ approximation, and a matching-based algorithm, which
achieves a $\max\{\frac{2}{3(2+\kappa)},\frac{2}{3}(1-\kappa)\}$ approximation (where
$\kappa$ is the curvature of the function).  Both
algorithms have a complexity of $O(|V|^3)$ in terms of number of
oracle calls.  We extended these results to directed graphs and
presented simulation results to empirically compare their performance
as well as evaluating the dependence on curvature.

There are several directions for future work.  First, we would like to
determine the tightness of the bounds that were presented.  The class
of submodular functions is very broad and so adding further
restrictions may help give a better idea of how the bounds change for
specific situations.
Another direction of research would be
considering extending other algorithms.  The strategies presented in
this paper are extensions of simple algorithms that are used to obtain
approximations for the traditional TSP. There are many other simple
strategies that could also be extended such as best neighbour or
insertion heuristics.  One other possible extension would be to
consider the case where multiple tours are needed (such as with
multiple patrolling robots).

%%%%%  

% Here we include the bib files
\bibliographystyle{IEEEtran}
\bibliography{alias,Main,New}

\begin{thebibliography}{10}
\providecommand{\url}[1]{#1}
\csname url@samestyle\endcsname
\providecommand{\newblock}{\relax}
\providecommand{\bibinfo}[2]{#2}
\providecommand{\BIBentrySTDinterwordspacing}{\spaceskip=0pt\relax}
\providecommand{\BIBentryALTinterwordstretchfactor}{4}
\providecommand{\BIBentryALTinterwordspacing}{\spaceskip=\fontdimen2\font plus
\BIBentryALTinterwordstretchfactor\fontdimen3\font minus
  \fontdimen4\font\relax}
\providecommand{\BIBforeignlanguage}[2]{{%
\expandafter\ifx\csname l@#1\endcsname\relax
\typeout{** WARNING: IEEEtran.bst: No hyphenation pattern has been}%
\typeout{** loaded for the language `#1'. Using the pattern for}%
\typeout{** the default language instead.}%
\else
\language=\csname l@#1\endcsname
\fi
#2}}
\providecommand{\BIBdecl}{\relax}
\BIBdecl

\bibitem{MLF-GLN-LAW:79}
M.~L. Fisher, G.~L. Nemhauser, and L.~A. Wolsey, ``An analysis of
  approximations for finding a maximum weight hamiltonian circuit,''
  \emph{Operations Research}, vol.~27, no.~4, pp. pp. 799--809, 1979.

\bibitem{RH-SR:98}
R.~Hassin and S.~Rubinstein, ``Better approximations for max {TSP},''
  \emph{Information Processing Letters}, vol.~75, pp. 181--186, 1998.

\bibitem{CG-AK-AS:05}
C.~Guestrin, A.~Krause, and A.~Singh, ``Near-optimal sensor placements in
  {G}aussian processes,'' in \emph{Int. Conf. on Machine Learning}, Bonn,
  Germany, Aug. 2005.

\bibitem{DG-AK:11}
D.~Golovin and A.~Krause, ``Adaptive submodularity: Theory and applications in
  active learning and stochastic optimization,'' \emph{Journal of Artiﬁcial
  Intelligence Research}, vol.~42, pp. 427--486, 2011.

\bibitem{PRG-ASS:07}
P.~R. Goundan and A.~S. Schulz, ``Revisiting the greedy approach to submodular
  set function maximization,'' 2007, {W}orking {P}aper, {M}assachusetts
  {I}nstitute of {T}echnology.

\bibitem{AS-AK-CG-WK:09}
A.~Singh, A.~Krause, C.~Guestrin, and W.~Kaiser, ``Efficient informative
  sensing using multiple robots,'' \emph{Journal of Artificial Intelligence
  Research}, vol.~34, pp. 707--755, 2009.

\bibitem{AS:00}
A.~Schrijver, ``A combinatorial algorithm minimizing submodular functions in
  strongly polynomial time,'' \emph{Journal of Combinatorial Theory, Series B},
  vol.~80, no.~2, pp. 346 -- 355, 2000.

\bibitem{SI-LF-SF:01}
S.~Iwata, L.~Fleischer, and S.~Fujishige, ``A combinatorial strongly polynomial
  algorithm for minimizing submodular functions,'' \emph{J. ACM}, vol.~48,
  no.~4, pp. 761--777, Jul. 2001.

\bibitem{GLN-LAW-MLF:78A}
G.~L. Nemhauser, L.~A. Wolsey, and M.~L. Fisher, ``An analysis of
  approximations for maximizing submodular set functions - {I},''
  \emph{Mathematical Programming}, vol.~14, pp. 265--294, 1978.

\bibitem{GC-CC-MP-JV:11}
G.~Calinescu, C.~Chekuri, M.~Pál, and J.~Vondrák, ``Maximizing a monotone
  submodular function subject to a matroid constraint,'' \emph{SIAM Journal on
  Computing}, vol.~40, no.~6, pp. 1740--1766, 2011.

\bibitem{GLN-LAW-MLF:78B}
M.~L. Fisher, G.~L. Nemhauser, and L.~A. Wolsey, ``An analysis of
  approximations for maximizing submodular set functions - {II},'' in
  \emph{Polyhedral Combinatorics}, ser. Mathematical Programming Studies, 1978,
  vol.~8, pp. 73--87.

\bibitem{JW:12}
J.~Ward, ``A (k+3)/2-approximation algorithm for monotone submodular k-set
  packing and general k-exchange systems,'' in \emph{29th International
  Symposium on Theoretical Aspects of Computer Science}, vol.~14, Dagstuhl,
  Germany, 2012, pp. 42--53.

\bibitem{MC-GC:84}
M.~Conforti and G.~Cornuejols, ``Submodular set functions, matroids and the
  greedy algorithm: Tight worst-case bounds and some generalizations of the
  rado-edmonds theorem,'' \emph{Discrete Applied Mathematics}, vol.~7, no.~3,
  pp. 251 -- 274, 1984.

\bibitem{BK-JV:07}
B.~Korte and J.~Vygen, \emph{Combinatorial Optimization: Theory and
  Algorithms}, 4th~ed., ser. Algorithmics and Combinatorics.\hskip 1em plus
  0.5em minus 0.4em\relax Springer, 2007, vol.~21.

\bibitem{MM:78}
M.~Minoux, ``Accelerated greedy algorithms for maximizing submodular set
  functions,'' in \emph{Optimization Techniques}, ser. Lecture Notes in Control
  and Information Sciences, J.~Stoer, Ed.\hskip 1em plus 0.5em minus
  0.4em\relax Springer Berlin / Heidelberg, 1978, vol.~7, pp. 234--243.

\bibitem{UF:98}
U.~Feige, ``A threshold of ln n for approximating set cover,'' \emph{J. ACM},
  vol.~45, no.~4, pp. 634--652, Jul. 1998.

\bibitem{JM:06}
J.~Mestre, ``Greedy in approximation algorithms,'' in \emph{Algorithms – ESA
  2006}, Y.~Azar and T.~Erlebach, Eds.\hskip 1em plus 0.5em minus 0.4em\relax
  Springer Berlin / Heidelberg, 2006, vol. 4168, pp. 528--539.

\bibitem{TAJ:79}
T.~A. Jenkyns, ``The greedy travelling salesman's problem,'' \emph{Networks},
  vol.~9, no.~4, pp. 363--373, 1979.

\bibitem{THC-CEL-RLR-CS:01}
T.~H. Cormen, C.~E. Leiserson, R.~L. Rivest, and C.~Stein, \emph{Introduction
  to Algorithms}, 2nd~ed.\hskip 1em plus 0.5em minus 0.4em\relax MIT Press,
  2001.

\end{thebibliography}
\end{document}